\documentclass
[
    a4paper,
    DIV=11,
    abstracton
]
{scrartcl}
\usepackage{fullpage,authblk,amsmath,amssymb,amsthm,amsfonts}
\usepackage{bm, bbm, dsfont}
\usepackage{enumerate}
\usepackage{url,mathtools}
\usepackage{enumitem}
\usepackage{xcolor}
\usepackage[pdffitwindow=true,
		   pdfstartview={FitH},
            plainpages=false,
            pdfpagelabels=false,
            pdfpagemode=UseOutlines,
            pdfpagelayout=SinglePage,
            bookmarks=false,
            colorlinks=true,
            hyperfootnotes=false,
            linkcolor=blue,
            urlcolor=blue!30!black,
            citecolor=green!50!black]{hyperref}
\usepackage{cleveref}

\theoremstyle{plain}
\newtheorem{theorem}{Theorem}[section]
\newtheorem{lemma}[theorem]{Lemma}
\newtheorem{proposition}[theorem]{Proposition}
\newtheorem{corollary}[theorem]{Corollary}
\newtheorem{assumption}[theorem]{Assumption}

\theoremstyle{definition}

\newtheorem{example}[theorem]{Example}
\newtheorem{remark}[theorem]{Remark}

\numberwithin{equation}{section} 

\newcommand{\ud}{\,\mathrm{d}}

\newcommand{\abs}[1]{\left\vert#1\right\vert}
\newcommand{\Vrt}[1]{\left\Vert #1 \right\Vert}
\newcommand{\Ang}[1]{\left\langle #1 \right\rangle}

\newcommand{\bb}[1]{\mathbb{#1}}

\newcommand{\filt}[0]{(\mathcal{F}_t)_{t\geq 0}}
\newcommand{\run}[0]{\textup{run}}
\newcommand{\term}[0]{\textup{term}}
\newcommand{\KLD}[2]{D_\textup{KL}(#1\Vert #2)}

\newcommand{\filtprobspace}[0]{(\Omega,\mathcal{F},(\mathcal{F}_t)_{t\geq 0},\mathbb{P})}

\begin{document}
\title{Fr\'{e}chet derivatives of expected functionals of solutions to stochastic differential equations}

\author{Han Cheng Lie\thanks{~Universit\"{a}t Potsdam, Institut f\"{u}r Mathematik, Karl-Liebknecht-Str. 24/25, D-14476, Potsdam, Germany.\ hanlie@uni-potsdam.de}}
\renewcommand\Affilfont{\small}
\affil[]{Department of Mathematics, Universit\"at Potsdam, Germany}

\date{}

\maketitle

\begin{abstract}
    In the analysis of stochastic dynamical systems described by stochastic differential equations (SDEs), it is often of interest to analyse the sensitivity of the expected value of a functional of the solution of the SDE with respect to perturbations in the SDE parameters. In this paper, we consider path functionals that depend on the solution of the SDE up to a stopping time. We derive formulas for Fr\'{e}chet derivatives of the expected values of these functionals with respect to bounded perturbations of the drift, using the Cameron-Martin-Girsanov theorem for the change of measure. Using these derivatives, we construct an example to show that the map that sends the change of drift to the corresponding relative entropy is not in general convex. We then analyse the existence and uniqueness of solutions to stochastic optimal control problems defined on possibly random time intervals, as well as gradient-based numerical methods for solving such problems. 
\end{abstract}

%

\section{Introduction}
\label{sec:intro}

Given a dynamical system that is described by a stochastic differential equation (SDE), one can use multiple quantities of interest to provide a high-level description of the system. These quantities of interest are often formulated as expected values of functionals of the trajectory of the system. A natural question is to analyse the sensitivity of this expected value to perturbations in the parameters that define the original system. In the context of stochastic analysis, sensitivity analysis of expected values of path functionals is often done using the Malliavin calculus, see e.g. \cite{Nualart2006,daPrato2014}. In mathematical finance for example, the Malliavin calculus was used to compute the sensitivity -- expressed in terms of Gateaux derivatives --  of the expected value of so-called `greeks' with respect to variations in the drift coefficient, the initial condition, and the diffusion coefficient \cite[Section 3]{Fournie1999}. 

This paper considers path functionals that depend on the solution of the SDE up to a stopping time. The main motivation for considering path functionals that depend on the solution up to a stopping time is to allow for functionals that depend on a random event, such as the first exit of a diffusion process from a bounded domain. Instead of using the Malliavin calculus, we use only the Cameron-Martin-Girsanov change of measure. The key assumption that we use for our results is the exponential integrability of the stopping time, up to some strictly positive scalar. 

Our first main result is a formula for the $n$-th order Fr\'{e}chet derivative of the expectation of any path functional that does not depend on the change of drift parameter, for any $n\in\bb{N}$. The formula recovers the integration by parts formula from Malliavin calculus. In addition, we state and prove a formula for the $n$-th order Fr\'{e}chet derivative of the expectation of the negative log Radon--Nikodym derivative, i.e. for the relative entropy or the expected quadratic cost of control. Since the negative log Radon--Nikodym derivative depends on the change of drift, this formula is not a consequence of the previous formula. However, the same proof strategy applies, and the formula turns out to be only slightly more complicated. These formulas are obtained under the assumption that the smallest nonzero singular value of the diffusion coefficient is bounded away from zero -- which is a weaker assumption than uniform ellipticity of the diffusion -- and the exponential integrability assumption on the stopping time that we mentioned earlier.

We apply the formulas for the Fr\'{e}chet derivatives to analyse the convexity properties of the map that sends the change of drift $u$ to the expected value of a path functional with respect to the law $\mu^{u}$ of the solution of the SDE with the perturbed drift. Using the formula for the Fr\'{e}chet derivative of the map that sends the change of drift to the relative entropy, we show by an example that this map is in general not convex. The example involves the application of the formula for the exit distribution of one-dimensional Brownian motion from a bounded interval containing 0. This example shows that, although the map that sends a pair of mutually absolutely continuous probability measures to their relative entropy is jointly convex in its arguments, the same is not in general true of the composition of this map with the map that sends the change of drift $u$ to the law $\mu^{u}$ of the solution of the SDE with perturbed drift.

The original motivation for this work is the analysis of a class of stochastic optimal control problems, where the constraint is given by the controlled SDE 
\begin{equation*}
 \ud X^{u}_t=(b+u)(t,X^{u}_t)\ud t+\sigma(t,X^{u}_t)\ud B_t,\quad X^{u}_0=x\in D
\end{equation*}
for suitable drift coefficient $b$, control or change of drift $u$, diffusion coefficient $\sigma$, standard $d$-dimensional Brownian motion $B$, and a bounded domain $D\subset\bb{R}^{d}$. The objective functional is
\begin{equation*}
 u\mapsto \bb{E}\left[\int_0^\tau \tilde{f}(s,X^{u}_s)+\frac{\lambda}{2}\abs{u(s,X^u_s)}^2\ud s+\tilde{g}(\tau,X^{u}_\tau)\right],
\end{equation*}
where $\tau$ is the first exit time from the bounded domain $D$, $\tilde{f}$ is a running cost function, and $\tilde{g}$ is a terminal cost function. The aim was to analyse the convergence of a gradient-based numerical method described in \cite{HartmannSchuette_2012} for solving problems in this class. In this paper, we achieve this by using the formulas for the derivatives to prove strict convexity of the objective functional, under suitable conditions. In particular, we can specify a set $U$ of admissible controls and sufficient conditions on the parameters such that the objective admits a unique minimiser $u^\ast \in U$. Note that we do not consider the Hamilton-Jacobi-Bellman equation associated to this optimal control problem, or the relationship between the unique minimiser $u^\ast$ in $U$ and the value function of the optimal control problem. 

To analyse the convexity of a map on a real Banach space taking values in the extended real line, it suffices to use the Gateaux derivatives of the map, provided they exist. This is because the convexity of such a map can be characterised in terms of the convexity of its restriction to the one-dimensional line segment between any two points in its domain. In \cite{Lie2016}, the first- and second-order Gateaux derivatives of the objective functional were derived and used to find sufficient conditions for the strict convexity of the objective functional. One important advantage of working with Fr\'{e}chet derivatives is that the proofs are simpler. In addition, we obtain the Fr\'{e}chet derivatives under weaker assumptions than those stated in \cite{Lie2016}.

Fr\'{e}chet derivatives are also useful in other contexts. In \cite{Koltai_2019}, it was shown that the Perron-Frobenius and Koopman operators associated to time-inhomogeneous SDEs and finite, deterministic time intervals were Fr\'{e}chet differentiable with respect to the drift. This linear response-type result was then applied to infinite-dimensional optimisation problems in the context of dynamical systems \cite{FroylandKoltaiStahn2020}.

The paper is organised as follows. In \cref{sec:setup} we introduce the `reference' SDE and other notation.  In \cref{sec:derivatives}, we present the formulas for the Fr\'{e}chet derivatives. We use these derivatives to analyse the convexity of the maps that send the change of drift to the expected value of certain path functionals in \cref{sec:convexity}. In \cref{sec:application_stochastic_optimal_control} we combine the results from the preceding two sections to investigate certain stochastic optimal control problems and Monte Carlo methods that use gradient information about the objective.  

\section{Setup}
\label{sec:setup}

In this section, we introduce some notation and basic assumptions. We fix a dimension $d\in\bb{N}$. For every $n\in\bb{N}$, $[n]:=\{1,\ldots,n\}$. We write $\bb{R}_{\geq 0}$ for the set $\{x\in\bb{R}\ :\  x\geq 0\}$. For a matrix $M\in\bb{R}^{m\times n}$ with $m,n\in\bb{N}$, $M^\top$, $M^+$ and $M^{-1}$ denote the transpose, pseudoinverse, and inverse respectively. For $1\leq p\leq \infty$, let $\abs{\cdot}_p$ denote the $\ell_p$ norm on Euclidean space. Given two normed vector spaces $V_1$ and $V_2$, $L(V_1;V_2)$ denotes the space of bounded, linear maps from $V_1$ to $V_2$. Given two topological spaces $X$ and $Y$, $C(X;Y)$ denotes the space of continuous functions from $X$ to $Y$. For arbitrary $\bb{R}$-valued random variables $X,Y$ defined on the same probability space, $\textup{Cov}(X,Y)$ and $\textup{Corr}(X,Y)$ denote the covariance and correlation respectively, and $\textup{Var}(X)$ denotes the variance of $X$. 

Let $\filtprobspace$ be a filtered probability space satisfying the usual conditions, and let $\mathbf{W}:=C(\bb{R}_{\geq 0};\bb{R}^d)$. Denote the set of coordinate mappings by $\{w(s)\ :\ s\geq 0\}$. For every $t\geq 0$, let $(\mathcal{B}_t)_{t\geq 0}$ denote the filtration on $\mathbf{W}$ generated by the coordinate mappings, i.e. $\mathcal{B}_t=\sigma(\{w(s)\ :\ s\leq t\})$ for every $t\geq 0$. A function $h$ on $\bb{R}_{\geq 0}\times\mathbf{W}$ is \emph{predictable} if $\{h(t,w)\ :\ t\in\bb{R}_{\geq 0}, w\in\mathbf{W}\}$ is a stochastic process that is predictable with respect to $(\mathcal{B}_t)_t$. For such a predictable function $h$, $s\geq 0$, and $w=(w(t))_{t\geq 0}\in\mathbf{W}$, we denote the value of $h$ at time $s$ on the path $w$ by $h(s,w_\bullet)$. An important special case is when there exists a function $\sigma$ on $\bb{R}_{\geq 0}\times \bb{R}^d$ such that $h(s,w)=\sigma(s,w(s))$ for every $w\in\mathbf{W}$. Then $h(s,X_\bullet)=\sigma(s,X_s)$.

Suppose we are given two predictable functions $f:\bb{R}_{\geq 0}\times\mathbf{W}\to\bb{R}^{d\times d}$ and $g:\bb{R}_{\geq 0}\times \mathbf{W}\to\bb{R}^d$. Consider the stochastic differential equation (SDE) with initial condition $X_0$, drift $g$, and diffusion $f$, represented by
\begin{equation}
\label{eq:sde}
X^0_t=X^0_0+\int_0^t g(s,X^0_\bullet)\mathrm{d}s+\int_0^t f(s,X^0_\bullet)\mathrm{d}B_s.
\end{equation}
Both $X^0$ and $B$ are $\bb{R}^d$-valued stochastic processes, while $B=(B_t)_{t\geq 0}$ is a standard $\filt$-Brownian motion. Throughout this paper, we shall make the following assumption:
\begin{assumption}
\label{asmp:exists_soln_to_sde_unique_in_law}
The coefficients $g$ and $f$ are such that there exists a solution $(X^0,B)$ on the filtered probability space $\filtprobspace$ to the SDE \eqref{eq:sde}, and this solution is unique in law. 
\end{assumption}
We use \cite[Chapter IX, Definitions 1.2 and 1.3]{RevuzYor2009} for the definition of a solution and for uniqueness in law. Next, let $(V,\Vrt{\cdot})$ be the normed vector space 
\begin{equation}
\begin{aligned}
 \label{eq:V_vector_space_of_changes_of_drift}
 V:=&\{h:\bb{R}_{\geq 0}\times\mathbf{W}\to\bb{R}^d\ :\ h\text{ is bounded and predictable}\}
 \\
 \Vrt{h}_V:=&\sup\{ \abs{h(t,w)}_2\ :\ (t,w)\in\bb{R}_{\geq 0}\times\mathbf{W}\}.
 \end{aligned}
\end{equation}

If $u\in V$, then given \cref{asmp:exists_soln_to_sde_unique_in_law}, it follows that
\begin{equation}
 \label{eq:sde_perturbed}
 X^u_t=X^u_0+\int_0^t (g+u)(s,X^u_\bullet)\mathrm{d}s+\int_0^t f(s,X^u_\bullet)\mathrm{d}B_s
\end{equation}
admits a solution $(X^u,B)$ that is unique in law, since the change of drift is bounded \cite[Chapter IX, Theorem 1.11]{RevuzYor2009}. 

For any continuous local martingale $M$ under $\bb{P}$, $\mathcal{E}(M)$ denotes the Dol\'{e}ans exponential martingale
\begin{equation}
\label{eq:exponential_martingale}
 \mathcal{E}(M)_t:=\exp\left(M_t-\tfrac{1}{2}\Ang{M,M}_t\right),\quad \forall t\geq 0.
\end{equation}
For any continuous local martingales $M=(M_t)_{t\geq 0}$ and $N=(N_t)_{t\geq 0}$ under $\bb{P}$, $\Ang{M,N}=(\Ang{M,N}_t)_{t\geq 0}$ is the associated covariance process that makes $(M_tN_t-\Ang{M,N}_t)_{t\geq 0}$ a continuous local martingale under $\bb{P}$. The Kunita-Watanabe inequality states that
\begin{equation}
 \label{eq:kunita_watanabe_inequality}
 \Ang{M,N}_t\leq \Ang{M,M}_t^{1/2}\Ang{N,N}_t^{1/2},\quad \forall t\geq 0.
\end{equation}
Next, define 
\begin{equation}
\label{eq:clmg}
 M^u_t(X,B):=\int_0^t \left(f^+ u(s,X_\bullet)\right)^\top\mathrm{d}B_s,\quad\forall t\geq 0,
\end{equation}
where $B$ is the standard $\bb{R}^d$-valued Brownian motion on $\filtprobspace$. We shall often write $M^u$ instead of $M^u(X,B)$. We shall use the linearity of the map $V\ni v\mapsto M^v$, i.e.
\begin{equation}
\label{eq:linearity_of_change_of_drift_to_CLMG}
M^{\lambda v+w}=\lambda M^{v}+M^{w},\quad \forall \lambda\in\bb{R},\ v,w\in V.
\end{equation}
In particular, for $u,v\in V$ and given \eqref{eq:clmg}, $\langle M^{u},M^{v}\rangle$ satisfies
\begin{equation}
\label{eq:quadratic_covariation_process}
\Ang{M^{u},M^{v}}_t(w_\bullet)=\int_0^t \left(u^\top(ff^\top)^{+}v\right)(s,w_\bullet)\ud s,\quad \forall t\geq 0,\ w\in\mathbf{W}.
\end{equation}
Given \eqref{eq:linearity_of_change_of_drift_to_CLMG}, the map $V\times V\ni (u,v)\mapsto \Ang{M^u,M^v}$ is symmetric and bilinear.

Suppose $u$ and $w\in V$ are such that, if we replace $u$ with $u+w$ in \eqref{eq:sde_perturbed}, then there exists a solution $(X^{u+w},B)$ and this solution is unique in law. Given a predictable function $\varphi:\bb{R}_{\geq 0}\times \mathbf{W}\times\mathbf{W}\to\bb{R}$ and a stopping time $\tau$ such that $\varphi_\tau(X^{u+w},B)\in L^1(\bb{P})$, we may use the exponential martingale in the reweighting formula
\begin{equation*}
 \bb{E}_{\bb{P}}[\varphi_\tau(X^{u+w},B)]=\bb{E}_{\bb{P}}[\varphi_\tau(X^u,B)\mathcal{E}(M^w(X^u,B))_\tau].
\end{equation*}
For any $u$ such that $\mu^{u}:=\bb{P}\circ (X^u,B)^{-1}$ exists and is unique, let $\bb{E}^{u}[\cdot]$ denote the associated expectation operator. We shall rewrite the reweighting formula above as
\begin{equation}
\label{eq:change_of_measure}
 \bb{E}^{u+w}[\varphi_\tau]=\bb{E}^{u}[\varphi_\tau\mathcal{E}(M^w)_\tau].
\end{equation}
For $p>0$, we define
\begin{equation*}
\Vrt{\widetilde{\varphi}}_{L^p(\mu^{u})}:=\bb{E}^{u}[\abs{\widetilde{\varphi}}^p]^{1/p},\quad L^p(\mu^{u}):=\{\widetilde{\varphi}:\mathbf{W}\times\mathbf{W}\to\bb{R}\ :\ \bb{E}^{u}[\abs{\widetilde{\varphi}}^p]^{1/p}<\infty\}.
\end{equation*}
Below, $\tau:\mathbf{W}\to\bb{R}_{\geq 0}$ denotes fixed, predictable $\filt$-stopping time that may be random or deterministic and that need not be bounded, $\phi:\mathbb{R}_{\geq 0}\times \mathbf{W}\to\bb{R}$ denotes a predictable function, and neither $\tau$ nor $\phi$ have any parametric dependence on the change of drift $u$ in \eqref{eq:sde_perturbed}. Unless we state otherwise, the path argument that we will give $\tau$ and $\phi$ is $X^u$, i.e. the first component of the solution to the SDE \eqref{eq:sde_perturbed}. For example, $\bb{E}^{u}[\phi_\tau]=\bb{E}_{\bb{P}}[\phi_{\tau(X^u)}(X^u)]$.

\section{Fr\'{e}chet derivatives}
\label{sec:derivatives}

In this section, we state a formula for the $n$-th order Fr\'{e}chet derivative of the map $u\mapsto \bb{E}^{u}[\phi_{\tau}]$ in \cref{lem:mixed_second_order_frechet_derivative_for_functional_not_depending_on_perturbation}. The formula recovers the integration by parts formula from Malliavin calculus. We also state a formula for the $n$-th order Fr\'{e}chet derivative of the map $u\mapsto \bb{E}^{u}[\Ang{M^u,M^u}_{\tau}]=\bb{E}^{u}[-\log\mathcal{E}(M^{-u})_\tau]$ in \cref{lem:mixed_second_order_frechet_derivative_for_KLD_mu_u_bar_mu_0_clmg_form}. These formulas are obtained under an assumption that is weaker than uniform ellipticity of the diffusion, and an exponential integrability assumption on the stopping time. We first state these assumptions below.
\begin{assumption}
\label{asmp:uniform_lower_bound_on_smallest_singular_value_of_diffusion}
 There exists a constant $\alpha>0$ such that 
 \begin{equation*}
  0\leq y^\top (ff^\top)^+(t,w_\bullet) y\leq \alpha^{-2} \abs{y}_2^2,\quad \forall y\in \bb{R}^d,\ w\in\mathbf{W},\ t\geq 0.
 \end{equation*}
\end{assumption}
In \cref{asmp:uniform_lower_bound_on_smallest_singular_value_of_diffusion},  we do not assume $f$ to be invertible. 
\begin{assumption}
 \label{asmp:ball_of_admissible_changes_of_drift}
There exists a convex set $U\subset V$ with $0\in U$ such that for every $u\in U$, there exists $\lambda_u>0$ such that $\bb{E}^{u}[\exp(\lambda_u\tau)]$ is finite. In particular, $\tau\in L^p(\mu^{u})$ for every $p\geq 1$.
\end{assumption}

The following result shows that if \cref{asmp:uniform_lower_bound_on_smallest_singular_value_of_diffusion} holds with parameter $\alpha$, and if there exists some $\lambda_0>0$ such that $\bb{E}^{0}[\exp(\lambda_0\tau)]$ is finite, then the set $U$ in \cref{asmp:ball_of_admissible_changes_of_drift} contains a ball centred at 0 with strictly positive radius that is proportional to $\lambda_0$. For the proof, see \cref{sec:proofs}.
\begin{lemma}
 \label{lem:existence_of_ball_of_admissible_changes_of_drift}
 Suppose that \cref{asmp:uniform_lower_bound_on_smallest_singular_value_of_diffusion} holds with parameter $\alpha$, and that there exists $\lambda_0>0$ such that $\bb{E}^{0}[\exp(\lambda_0\tau)]$ is finite. Let
 \begin{equation*}
  U_0:=\left\{v\in V\ :\ \Vrt{v}_V^2 \leq \lambda_0 \left(\tfrac{1}{4}\left(\tfrac{\sqrt{2}}{\sqrt{2}-1}\right)^2\alpha^{-2}\right)^{-1}\right\}.
 \end{equation*}
 Then $U_0\subset U$. In particular, if $\tau$ is a deterministic, fixed stopping time $T>0$, then in \cref{asmp:ball_of_admissible_changes_of_drift} we may choose $U=V$.
\end{lemma}

Next, we state some basic results.
\begin{lemma}
\label{lem:uniform_bound_on_quadratic_variation_time_t}
 Let $\alpha>0$ be as in \cref{asmp:uniform_lower_bound_on_smallest_singular_value_of_diffusion}. Then for every $u,v\in V$, and for any predictable $\filt$-stopping time $\tau$,
\begin{equation*}
 \Ang{M^{u},M^{v}}_{\tau}\leq  \alpha^{-2}\Vrt{u}_V\Vrt{v}_V \tau.
\end{equation*}
If in addition \cref{asmp:ball_of_admissible_changes_of_drift} holds, then for every $r\geq 2$,
\begin{align}
 \Vrt{\Ang{M^u,M^v}_\tau}_{L^r(\mu^{u})} &\leq \alpha^{-2}\Vrt{u}_V\Vrt{v}_V\Vrt{\tau}_{L^r(\mu^{u})}
 \label{eq:Lr_bound_covariance}
 \\
 \Vrt{M^v_\tau}_{L^r(\mu^{u})} &\leq 2\alpha^{-1}\Vrt{v}_V \Vrt{\tau}_{L^{r/2}(\mu^{u})}^{1/2}.
 \label{eq:Lr_bound_clmg}
\end{align}
\end{lemma}
\begin{proof}[Proof of \cref{lem:uniform_bound_on_quadratic_variation_time_t}]
Given \eqref{eq:quadratic_covariation_process}, the definition of $\Vrt{\cdot}_V$ in \eqref{eq:V_vector_space_of_changes_of_drift}, and  \cref{asmp:uniform_lower_bound_on_smallest_singular_value_of_diffusion},  
\begin{equation*}
\Ang{M^u,M^u}_{\tau}(w)=\int_0^\tau \left(u^\top (ff^\top)^+ u\right)(s,w_\bullet)\ud s \leq \alpha^{-2}\Vrt{u}^2_V\tau,\quad \forall w\in\mathbf{W}. 
\end{equation*}
By the Kunita-Watanabe inequality \eqref{eq:kunita_watanabe_inequality}, the first bound \eqref{eq:Lr_bound_covariance} follows. Taking expectations yields \eqref{eq:Lr_bound_covariance}. For $r>1$ and for any continuous local martingale $M$ on $\filtprobspace$, Doob's inequality states 
	\begin{equation*}
	\bb{E}_{\bb{P}}\left[ \sup_{0\leq t\leq \tau}\abs{M_t}^{r}\right]^{1/r}\leq \left(\frac{r}{r-1}\right) \bb{E}_{\bb{P}}\left[ \Ang{M,M}_\tau^{r/2}\right]^{1/r}.
	\end{equation*}
	Combining Doob's inequality with the fact that $r\mapsto (r/(r-1))$ is decreasing on the interval $(1,\infty)$ implies that
	\begin{equation}
	 \label{eq:doobs_inequality}
	 \Vrt{M^v_\tau}_{L^r(\mu^{u})}\leq 2\Vrt{\Ang{M^v,M^v}_\tau}_{L^{r/2}(\mu^{u})}^{1/2},\quad\forall r\geq 2,
	\end{equation}
    and combining \eqref{eq:doobs_inequality} with \eqref{eq:Lr_bound_covariance} yields \eqref{eq:Lr_bound_clmg}.
\end{proof}

The next result is crucial for finding the formulas of Fr\'{e}chet derivatives. Its proof is given in \cref{sec:proofs}.
\begin{lemma}
\label{lem:convergence_lemmas_for_change_of_measure}
Suppose that  \cref{asmp:uniform_lower_bound_on_smallest_singular_value_of_diffusion} and \cref{asmp:ball_of_admissible_changes_of_drift} hold. Let $u\in U$. Then for every $s\geq 1$ and $0\leq r<1$,
\begin{align}
\lim_{\Vrt{v}_V\to 0}\frac{ \Vrt{ \mathcal{E}(M^v)_\tau-1-M^v_\tau}_{L^s(\mu^{u})}}{\Vrt{v}_V}=&  0.
 \label{eq:convergence_lemma_for_change_of_measure_minus_one_minus_CLMG}
 \\
\lim_{\Vrt{v}_V\to 0}\frac{ \Vrt{ \mathcal{E}(M^v)_\tau-1}_{L^s(\mu^{u})}}{\Vrt{v}_V^{r}}= & 0.
 \label{eq:convergence_lemma_for_change_of_measure_minus_one}
\end{align}
\end{lemma}

With all the preparations in place, we can now state the formula for the first Fr\'{e}chet derivative of the map that sends the change of drift to the expectation of a functional of the controlled diffusion, where the functional does not exhibit parametric dependence on the change of drift. By \eqref{eq:clmg}, $M^v_\tau$ is an It\^{o} integral. Hence, the formula for the Fr\'{e}chet derivative agrees with the stochastic integration by parts formula from the Malliavin calculus. For example, one can compare the result below with \cite[Proposition 3.1]{Fournie1999} or \cite[Lemma 1.2.1]{Nualart2006}.
\begin{lemma}
\label{lem:frechet_derivative_for_functional_not_depending_on_perturbation}
Suppose that \cref{asmp:uniform_lower_bound_on_smallest_singular_value_of_diffusion} and \cref{asmp:ball_of_admissible_changes_of_drift} hold. If $\phi_\tau\in L^2(\mu^{u})$, then the Fr\'{e}chet derivative of the map $U\ni u'\mapsto \bb{E}^{u'}[\phi_\tau]$ at $u$ is given by the linear map $V\ni v\mapsto \bb{E}^{u}[\phi_\tau M^{v}_\tau]$.
\end{lemma}
\begin{proof}[Proof of \cref{lem:frechet_derivative_for_functional_not_depending_on_perturbation}]
Using \eqref{eq:change_of_measure} and the Cauchy-Schwarz inequality,
\begin{align*}
 \abs{\bb{E}^{u+v}[\phi_\tau]-\bb{E}^{u}[\phi_\tau]-\bb{E}^{u}[\phi_\tau M^{v}_\tau]} =&\abs{\bb{E}^{u}\left[\phi_\tau\left(\mathcal{E}(M^v)_\tau-1-M^v_\tau\right)\right]}
 \\
 \leq & \Vrt{\phi_\tau}_{L^2(\mu^{u})}\Vrt{\mathcal{E}(M^{v})_\tau-1-M^{v}_\tau}_{L^2(\mu^{u})}.
\end{align*}
Thus, by using \eqref{eq:convergence_lemma_for_change_of_measure_minus_one_minus_CLMG} in \cref{lem:convergence_lemmas_for_change_of_measure}, 
\begin{equation*}
\lim_{\Vrt{v}_V\to 0}\frac{\abs{\bb{E}^{u+v}[\phi_\tau]-\bb{E}^{u}[\phi_\tau]-\bb{E}^{u}[\phi_\tau M^{v}_\tau]}}{\Vrt{v}_V}\leq
\Vrt{\phi_\tau}_{L^2(\mu^u)} \lim_{\Vrt{v}_V\to 0}\frac{\Vrt{\mathcal{E}(M^{v})_\tau-1-M^{v}_\tau}_{L^2(\mu^{u})}}{\Vrt{v}_V}=0.
\end{equation*}
By \eqref{eq:clmg}, the map $v\mapsto M^{v}_\tau$ is linear, and hence so is $v\mapsto \bb{E}^{u}[\phi_\tau M^{v}_\tau]$. Using the Cauchy-Schwarz inequality and \eqref{eq:Lr_bound_clmg}, we have
\begin{equation*}
 \abs{\bb{E}^{u}[\phi_\tau M^{v}_\tau]}\leq \Vrt{\phi_\tau}_{L^2(\mu^{u})}\Vrt{M^v_\tau}_{L^2(\mu^{u})} \leq \Vrt{\phi_\tau}_{L^2(\mu^{u})} \alpha^{-1}\Vrt{\tau}_{L^1(\mu^{u})}^{1/2} \Vrt{v}_V,\quad \forall v\in V,
\end{equation*}
which proves that $v\mapsto \bb{E}^{u}[\phi_\tau M^{v}_\tau]$ is bounded.
\end{proof}
To obtain the second Fr\'{e}chet derivative of the map $u'\mapsto \bb{E}^{u'}[\phi_\tau]$, we use the following result. 
\begin{lemma}
 \label{lem:mixed_second_order_frechet_derivative_for_functional_not_depending_on_perturbation}
 Suppose that \cref{asmp:uniform_lower_bound_on_smallest_singular_value_of_diffusion} and \cref{asmp:ball_of_admissible_changes_of_drift} hold. Let $u\in U$ and $n\in\bb{N}$. If $\phi_\tau\in L^2(\mu^{u})$ and $\phi_\tau$ is $\mu^{u}$-almost surely nonconstant, then the $n$-th order Fr\'{e}chet derivative of $U\ni u'\mapsto \bb{E}^{u'}[\phi_\tau ]$ at $u$ is given by the $n$-linear map 
 \begin{equation*}
   \times_{k=1}^{n} V\ni (v_1,\ldots,v_n)\mapsto \bb{E}^{u}\left[\phi_\tau \prod_{k=1}^{n}M^{v_k}_{\tau}\right].
 \end{equation*}
\end{lemma}
If we apply the formula above for some $\phi_\tau$ that $\mu^{u}$-almost surely equal to some nonzero constant $c\in\bb{R}$, then we obtain $c\cdot\bb{E}^{u}[\prod_{k=1}^{n}M^{v_k}_{\tau}]$, which need not be zero. For example, if $n=2$ and $v_1=v_2\in V\setminus \{0\}$ is such that $\mu^{u}(\Ang{M^{v_1},M^{v_1}}_\tau>0)>0$, then by the It\^{o} isometry, $c\cdot\bb{E}^{u}[\prod_{k=1}^{n}M^{v_k}_{\tau}]=c\cdot\bb{E}^{u}[\Ang{M^{v_1},M^{v_1}}_\tau]\neq 0$. This produces a contradiction with the fact that all Fr\'{e}chet derivatives of $u\mapsto c$ must vanish. For this reason, we introduce the assumption that $\phi_\tau$ be $\mu^{u}$-almost surely nonconstant in \cref{lem:mixed_second_order_frechet_derivative_for_functional_not_depending_on_perturbation}.
\begin{remark}
\label{rem:placeholder}
 In the proof of \cref{lem:mixed_second_order_frechet_derivative_for_functional_not_depending_on_perturbation}, we do not consider the case where one or more of the $(v_k)_{k=1}^{n}$ are equal to $u$. This is because the $(v_k)_{k=1}^{n}$ are only `placeholder' vectors in $V$. In other words, the $(v_k)_{k=1}^{n}$ should not be considered as fixed vectors in $V$, because the $(n+1)$-th order Fr\'{e}chet derivative is the derivative of the $n$-th order Fr\'{e}chet derivative, considered as an $n$-linear map on $\times_{k=1}^{n}V$.
\end{remark}

\begin{proof}[Proof of \cref{lem:mixed_second_order_frechet_derivative_for_functional_not_depending_on_perturbation}]
We prove the claim by induction. The case $n=1$ is given in \cref{lem:frechet_derivative_for_functional_not_depending_on_perturbation}. Let $v_{n+1}\in V$ be arbitrary, and suppose that the claim is true for some $n\leq 1$. Then by \eqref{eq:change_of_measure},
\begin{align*}
 &\abs{\bb{E}^{u+v_{n+1}}\left[\phi_\tau \prod_{k=1}^{n} M^{v_k}_{\tau}\right]-
 \bb{E}^{u}\left[\phi_\tau \prod_{k=1}^{n} M^{v_k}_{\tau}\right]
 -\bb{E}^{u}\left[\phi_\tau \prod_{k=1}^{n} M^{v_k}_{\tau} M^{v_{n+1}}_{\tau}\right]}
 \\
 =& \abs{\bb{E}^{u}\left[ \phi_\tau \prod_{k=1}^{n} M^{v_k}_{\tau}\left(\mathcal{E}(M^{v_{n+1}})_{\tau}-1-M^{v_{n+1}}_{\tau}\right)\right]}
 \\
 \leq & \Vrt{\phi_\tau}_{L^2(\mu^{u})}\Vrt{ \prod_{k=1}^{n} M^{v_k}_{\tau}}_{L^4(\mu^{u})}\Vrt{\mathcal{E}(M^{v_{n+1}})_\tau-1-M^{v_{n+1}}_{\tau}}_{L^4(\mu^{u})}.
\end{align*}
Dividing both sides of the inequality by $\Vrt{v_{n+1}}_V$, taking the limit as $\Vrt{v_{n+1}}_V\to 0$, and using \eqref{eq:convergence_lemma_for_change_of_measure_minus_one_minus_CLMG} yields the desired limit relation. The $n$-linearity follows from the linearity of the map $v_i\mapsto M^{v_i}_{\tau}$ for every $i\in\{1,\ldots,n+1\}$, which follows from \eqref{eq:clmg}. Boundedness of the map follows by H\"{o}lder's inequality and \eqref{eq:Lr_bound_clmg}:
\begin{align*}
\abs{\bb{E}^{u}\left[\phi_\tau \prod_{k=1}^{n} M^{v_k}_{\tau}\right]}&\leq \Vrt{\phi_{\tau}}_{L^2(\mu^{u})}  \Vrt{\prod_{k=1}^{n} M^{v_k}_{\tau}}_{L^2(\mu^{u})}\leq \Vrt{\phi_{\tau}}_{L^2(\mu^{u})}\prod_{k=1}^{n} \Vrt{M^{v_k}_{\tau}}_{L^{2n}(\mu^{u})}
\\
&\leq \Vrt{\phi_{\tau}}_{L^2(\mu^{u})}\left(2\alpha^{-1}\Vrt{\tau}_{L^{n}(\mu^{u})}^{1/2}\right)^{n}\prod_{k=1}^{n}\Vrt{v_k}_V.
\end{align*}

\end{proof}

So far we have considered Fr\'{e}chet derivatives of maps of the form $u\mapsto \bb{E}^{u}[\phi_\tau]$, where the functional $\phi_\tau$ has no parametric dependence on $u$. Next, we consider the map $u\mapsto \bb{E}^{u}[\Ang{M^u,M^u}_\tau]$. Since the functional $\Ang{M^u,M^u}_\tau$ exhibits parametric dependence on $u$, we cannot apply \cref{lem:frechet_derivative_for_functional_not_depending_on_perturbation} or \cref{lem:mixed_second_order_frechet_derivative_for_functional_not_depending_on_perturbation}. One motivation for considering this functional is that if $u\in U$ and $w\in V$ are such that $\mu^{u}$ and $\mu^{u+w}$ are locally equivalent, then the relative entropy or Kullback-Leibler divergence of $\mu^{u}$ with respect to $\mu^{u+w}$ on $\mathcal{F}_\tau$ satisfies
\begin{equation}
\label{eq:KLD}
\KLD{\mu^{u}}{\mu^{u+w}}\vert_{\mathcal{F}_\tau}=\bb{E}^{u}[-\log\mathcal{E}(M^{w})_\tau]= \tfrac{1}{2}\bb{E}^{u}\left[\Ang{M^w,M^w}_\tau\right]=\tfrac{1}{2}\bb{E}^{u}\left[(M^w_\tau)^2\right],
\end{equation}
by \eqref{eq:change_of_measure}, \eqref{eq:exponential_martingale},  \eqref{eq:quadratic_covariation_process}, and the It\^{o} isometry. The relevance of the relative entropy can be seen as follows. If we interpret $X^u_t$ as the position of a particle at time $t$ whose velocity changes according to an ambient force $g$ -- e.g. the force acting a particle by its environment -- and a control force $u$, then by \eqref{eq:quadratic_covariation_process}, one can view the relative entropy term as the total `kinetic energy' cost of the control $u$. Thus, in problems of stochastic optimal control, the presence of the relative entropy term in the objective ensures that the objective cannot be minimised or optimised by using arbitrarily large control. 

The following two results are analogues of \cref{lem:frechet_derivative_for_functional_not_depending_on_perturbation} and \cref{lem:mixed_second_order_frechet_derivative_for_functional_not_depending_on_perturbation}, and their proofs follow similar steps. For this reason, we postpone the proofs to \cref{sec:proofs}.

\begin{lemma}
 \label{lem:frechet_derivative_for_KLD_mu_u_bar_mu_0_clmg_form}
Suppose \cref{asmp:uniform_lower_bound_on_smallest_singular_value_of_diffusion} and \cref{asmp:ball_of_admissible_changes_of_drift} hold. The Fr\'{e}chet derivative of $U\ni u'\mapsto \bb{E}^{u'}[(M^{u'}_\tau)^2]$ at $u$ is given by
 \begin{equation*}
   V\ni v\mapsto \bb{E}^{u}\left[\left(\left(M^{u}_\tau\right)^2+2M^{u}_{\tau}\right)M^{v}_{\tau}\right].
 \end{equation*}
 \end{lemma}
 The main difference between \cref{lem:frechet_derivative_for_KLD_mu_u_bar_mu_0_clmg_form} and the corresponding result \cref{lem:frechet_derivative_for_functional_not_depending_on_perturbation} is that the \emph{sum} of the path functional $(M^u_\tau)^2$ and its Fr\'{e}chet derivative $2M^u_\tau$ must be weighted by the martingale term $M^v_\tau$. This indicates the rule of thumb for computing Fr\'{e}chet derivatives of expected values of functionals: one must first apply the Fr\'{e}chet derivative to the functional itself, and then multiply the sum of the original functional with its Fr\'{e}chet derivative by the martingale term associated to the direction $v$.
 
 \Cref{lem:frechet_derivative_for_KLD_mu_u_bar_mu_0_clmg_form} is the base case for the proof by induction of the following result. 
 \begin{lemma}
 \label{lem:mixed_second_order_frechet_derivative_for_KLD_mu_u_bar_mu_0_clmg_form}
 Suppose that \cref{asmp:uniform_lower_bound_on_smallest_singular_value_of_diffusion} and \cref{asmp:ball_of_admissible_changes_of_drift} hold. Let $u\in U$. For $n\in\bb{N}$, the $n$-th order Fr\'{e}chet derivative of $U\ni u'\mapsto \bb{E}^{u'}[(M^{u'}_\tau)^2]$ at $u$ is given by 
\begin{equation*}
 \times_{k=1}^{n}V\ni (v_1,\ldots,v_n)\mapsto \bb{E}^{u}\left[ \left(\left(M^u_\tau\right)^2+2n M^u_\tau +n(n-1)\right)\prod_{k=1}^{n}M^{v_k}_\tau\right].
\end{equation*}
\end{lemma}
Note that the same rule of thumb applies above: the $n(n-1)$ terms arise from taking Fr\'{e}chet derivative of $2M^u_\tau$. See \cref{lem:frechet_derivative_for_product_of_Mu_clmg_with_Z} in \cref{sec:proofs}.

\section{Convexity}
\label{sec:convexity}

In this section, we use the second Fr\'{e}chet derivatives of the maps $u\mapsto \bb{E}^{u}[\phi_\tau]$ and $u\mapsto \bb{E}^{u}[(M^u_\tau)^2]$ identified in \cref{sec:derivatives} to describe the convexity or nonconvexity of these maps. Recall that Fr\'{e}chet differentiability implies Gateaux differentiability, and that if the Fr\'{e}chet derivative exists, then it coincides with the Gateaux derivative. Recall the Banach space version of the second derivative test for convexity, as stated in \cite[Corollary 3.8.6]{NiculescuPersson2018}, for example.
 \begin{lemma}
  \label{lem:second_derivative_test_of_convexity}
  Let $V$ be a real Banach space and $U\subset V$ be convex and open. Suppose that $f:U\to \bb{R}$ is twice Gateaux differentiable with second-order Gateaux derivative $f''$. 
  \begin{enumerate}
   \item[(a)] $f$ is convex on $U$ if and only if for all $u\in U$ and $v\in V$ it holds that
   \begin{equation*}
    f''(u;v,v)\geq 0.
   \end{equation*}
   \item[(b)] If the above inequality is strict for every $v\in V\setminus\{0\}$, then $f$ is strictly convex.
  \end{enumerate}
 \end{lemma}
 We will also use the fact that convexity of a map is invariant under translations. We record this fact in the following lemma.
 
 \begin{lemma}
  \label{lem:translation_invariance_convex_maps}
  Let $V$ be a real Banach space and $U\subset V$ be convex. If $f:U\to\bb{R}$ is convex (respectively, strictly convex), then for every constant $c\in\bb{R}$, $(f+c):U\to\bb{R}$ is convex (resp. strictly convex).
 \end{lemma}
 \begin{proof}
  The statement for convexity follows from the fact that the sum of two convex functions is convex and from the fact that constant functions are convex. The statement for strict convexity follows from the fact that if at least one of two convex functions is strictly convex, then the sum of these functions is strictly convex. See \cite[Remark 3.1.6 and Proposition 1.1.10]{NiculescuPersson2018}. 
 \end{proof}

Next, we state the following assumptions. The first is a uniform ellipticity condition.
 \begin{assumption}
\label{asmp:uniform_ellipticity}
 The diffusion coefficient $f$ in \eqref{eq:sde} takes values in $\bb{R}^{d\times d}$, is invertible, and admits a constant $0<\alpha<\infty$ such that 
 \begin{equation*}
   y^\top (ff^\top)^{-1}(t,w_\bullet) y\leq \alpha^2 \abs{y}_2^2,\quad \forall y\in \bb{R}^d,\ w\in\mathbf{W},\ t\geq 0.
 \end{equation*}
\end{assumption}
Since \cref{asmp:uniform_ellipticity} involves the inverse $(ff^\top)^{-1}$ and not the pseudoinverse $(ff^\top)^+$ of $ff^\top$, it is stronger than  \cref{asmp:uniform_lower_bound_on_smallest_singular_value_of_diffusion}.

In addition, we make the following assumption on the predictable process $\phi$ and the stopping time $\tau$ that we introduced at the end of \cref{sec:setup}.
\begin{assumption}
 \label{asmp:uniform_lower_bound_on_path_functional}
  The predictable process $\phi$ and the stopping time $\tau$ admit a constant $c\in\bb{R}$ such that $\phi_{\tau(w)}(w)\geq c$ for every $w\in\mathbf{W}$.
\end{assumption}
\Cref{asmp:uniform_lower_bound_on_path_functional} is satisfied whenever the predictable function $\phi$ is bounded from below by $c$, for example. In \cref{sec:application_stochastic_optimal_control} we show a setting where \cref{asmp:uniform_lower_bound_on_path_functional} is satisfied.
 \begin{proposition}
 \label{prop:convexity_of_functional_not_depending_on_perturbation}
  Suppose \cref{asmp:ball_of_admissible_changes_of_drift} holds. Suppose that for every $u \in U$, $\phi_\tau\in L^2(\mu^{u})$ and is $\mu^{u}$-almost surely nonconstant. If both \cref{asmp:uniform_ellipticity} and \cref{asmp:uniform_lower_bound_on_path_functional} hold, then $u\mapsto \bb{E}^{u}[\phi_\tau]$ is strictly convex on $U$. 
 \end{proposition}
\begin{proof}[Proof of \cref{prop:convexity_of_functional_not_depending_on_perturbation}]
Since $\phi_\tau\in L^2(\mu^{u})$ and $\phi_\tau$ is $\mu^{u}$-almost surely nonconstant for every $u\in U$, then we may apply \cref{lem:mixed_second_order_frechet_derivative_for_functional_not_depending_on_perturbation} to conclude that the second Fr\'{e}chet derivative of $u\mapsto \bb{E}^{u}[\phi_\tau]$ at $u\in U$ satisfies $(v,v)\mapsto \bb{E}^{u} [\phi_\tau (M^v_\tau)^2]$. 

First assume that $c>0$ in \cref{asmp:uniform_lower_bound_on_path_functional}. It follows that for $v\in V\setminus\{0\}$,
\begin{equation*}
\bb{E}^{u} [\phi_\tau (M^v_\tau)^2]\geq c\cdot\bb{E}^{u}[(M^v_\tau)^2]=c\cdot\bb{E}^{u}[\Ang{M^v,M^v}_\tau]>0, 
\end{equation*}
where we used \cref{asmp:uniform_lower_bound_on_path_functional} in the first inequality and the It\^{o} isometry in the equation. For the strict inequality, we use the formula for $\Ang{M^v,M^v}$ in \eqref{eq:quadratic_covariation_process} and the uniform ellipticity of $f$ in \cref{asmp:uniform_ellipticity} to conclude that $\mu^{u}(\Ang{M^v,M^v}_\tau>0)=1$ for every $u\in U$, and hence $\bb{E}^{u}[\Ang{M^v,M^v}_\tau]>0$. By part (b) of \cref{lem:second_derivative_test_of_convexity}, the conclusion follows.

Now suppose that the constant in \cref{asmp:uniform_lower_bound_on_path_functional} satisfies $c\leq 0$. Then there exists some $c'>0$ such that $\hat{\phi}_\tau:=\phi_\tau+c'$ is uniformly bounded from below on $\mathbf{W}$ by a strictly positive number. By the argument from the preceding paragraph, the map $u\mapsto \bb{E}^{u}[\hat{\phi}_\tau]$ is strictly convex. Thus, $u\mapsto \bb{E}^{u}[\phi_\tau]$ is the sum of the strictly convex function $u\mapsto \bb{E}^{u}[\hat{\phi}_\tau]$ with the constant $-c'$. It follows from the translation invariance of convexity (\cref{lem:translation_invariance_convex_maps}) that $u\mapsto \bb{E}^{u}[\phi_\tau]$ is strictly convex.
\end{proof}

\begin{remark}
 If in \cref{prop:convexity_of_functional_not_depending_on_perturbation} we remove the assumption that $\phi_\tau$ is $\mu^{u}$-almost surely nonconstant for every $u\in U$, then there may exist a subset $U'$ of $U$ for which $u\mapsto \bb{E}^{u}[\phi_\tau]$ is constant on $U'$, and hence is only convex but not strictly convex. Thus, the assumption that $\phi_\tau$ is $\mu^{u}$-almost surely nonconstant for every $u\in U$ is necessary for strict convexity of $u\mapsto \bb{E}^{u}[\phi_\tau]$ on $U$. 
 
 If we use \cref{asmp:uniform_lower_bound_on_smallest_singular_value_of_diffusion} instead of \cref{asmp:uniform_ellipticity}, then the diffusion coefficient $f$ in \eqref{eq:sde_perturbed} may be singular. If the image of $\bb{R}_{\geq 0}\times \mathbf{W}$ under $v\in V$ is a subset of the nullspace of $f$, then $f^+v=0$ on $\bb{R}_{\geq 0}\times\mathbf{W}$. By the formula \eqref{eq:clmg} this implies that $\mu^{u}(\forall t\geq 0,\ M^v_t=0)=1$, and hence $\mu^{u}(\Ang{M^v,M^v}_\tau=0)=1$. In this case $u\mapsto \bb{E}^{u}[\phi_\tau]$ will not be strictly convex.
\end{remark}

Next, we consider the map $u\mapsto \bb{E}^{u}[(M^u_\tau)^2]$. By setting $w=v$ in \cref{lem:mixed_second_order_frechet_derivative_for_KLD_mu_u_bar_mu_0_clmg_form}, it follows that the second Fr\'{e}chet derivative of this map evaluated at $u\in U$ along $(v,v)\in V\times V$ equals
\begin{equation}
 \label{eq:second_order_frechet_derivative_for_KLD_mu_u_bar_mu_0_covariation_form}
 \bb{E}^{u}\left[(M^v_\tau)^2\left((M^u_\tau)^2 +4 M^u_\tau+2\right)\right]=:\bb{E}^{u}\left[ (M^v_\tau)^2 p(M^u_\tau)\right],
\end{equation} 
where we define $p(x):=x^2+4x+2$. Since $p(\cdot)$ attains its minimum value of $-2$ at $x^\ast=-2$, it follows that
\begin{equation}
\label{eq:second_order_frechet_derivative_for_KLD_mu_u_bar_mu_0_covariation_form_lower_bound}
 \bb{E}^{u}\left[(M^v_\tau)^2\left((M^u_\tau)^2 +4 M^u_\tau+2\right)\right]\geq -2\bb{E}^{u}[(M^v_\tau)^2].
\end{equation}
By the equivalent condition for convexity in statement (a) of \cref{lem:second_derivative_test_of_convexity}, the bound \eqref{eq:second_order_frechet_derivative_for_KLD_mu_u_bar_mu_0_covariation_form_lower_bound} suggests that the map $u\mapsto \bb{E}^{u}[(M^u_\tau)^2]$ may be nonconvex at some $u\in U$. The following example confirms this. 
\begin{example}
 \label{exa:change_of_drift_to_KLD_map_nonconvex}
  Consider the SDE \eqref{eq:sde_perturbed} in $\bb{R}^d$ with $d=1$, constant drift $g= 0$, constant diffusion $f=1$, and deterministic initial condition $X_0=0$. \Cref{asmp:uniform_ellipticity} holds with $\alpha=1$, because $f=1$. 
  
  Let $V$ be as in \eqref{eq:V_vector_space_of_changes_of_drift}, and let $U\subset V$ be the set of all constant $\bb{R}$-valued functions on $\bb{R}_{\geq 0}\times\mathbf{W}$. Then for every $u\in U$,
 \begin{equation*}
  X^u_t =(u-1)t+B_t,\quad \forall t\geq 0,
 \end{equation*}
 so that for every $u\in U$, $X^u$ is a standard Brownian motion under $\bb{P}$ with drift. From \eqref{eq:clmg} and the SDE above it follows that  $ M^u_t(X^u,B)=uB_t=u(X^u_t-(u-1)t)$ for every $t\geq 0$. Setting $u=1$ yields
 \begin{equation}
 \label{eq:special_choice_u}
 M^u_t(X^u,B)=B_t=X^u_t,\quad\forall t\geq 0.
 \end{equation}
 Now let $b>0$ be arbitrary, and let
 \begin{equation*}
  \tau:\mathbf{W}\to\bb{R}_{\geq 0},\quad  w\mapsto \tau(w):=\inf\{t>0\ :\ w_t\notin (-2,b)\}
 \end{equation*}
 be the first exit time from the interval $(-2,b)$. Using the Markov property, the fact that $B_0=0$, and $b>0$, it follows that
 \begin{equation*}
  \bb{P}(\tau(B)>t)=\bb{P}(B_s\in (-2,b)\text{ for all }s\in [0,t])\leq \left(\max_{x\in (-2,b)}\bb{P}(x+B_1\in (-2,b))\right)^k,\quad t\geq k\in\bb{N},
 \end{equation*}
 see e.g. the proof of \cite[Theorem 2.49]{MortersPeres2010}. Let $\beta:=-\log \max_{x\in (-2,b)}\bb{P}(x+B_1\in (-2,b))$, and note that $0<\beta<\infty$. By the bound above, there exists $C>0$ such that 
 \begin{equation*}
  \bb{P}(\tau(B)>t)\leq C\exp(-\beta t),\quad t>0.
 \end{equation*}
 By the tail probability formula and the inequality above,
 \begin{equation*}
  \bb{E}_{\bb{P}}[\exp(\lambda_u \tau(B))]=\int_0^\infty \lambda_u \exp(\lambda_u t)\bb{P}(\tau(B)>t)\ud t\leq C\lambda_u\int_0^\infty \exp(-(\beta-\lambda_u)t)\ud t,
 \end{equation*}
 which is finite whenever $\beta-\lambda_u>0$. Since $U$ is convex and contains 0, this proves that \cref{asmp:ball_of_admissible_changes_of_drift} is satisfied.

 By definition of $\tau$, it follows that 
 \begin{equation}
 \label{eq:distribution_first_exit_location}
  \bb{P}(B_{\tau(B)}=-2)=\tfrac{b}{2+b},\quad \bb{P}(B_{\tau(B)}=b)=\tfrac{2}{2+b},
 \end{equation}
 see e.g. \cite[Theorem 2.49]{MortersPeres2010}. Setting $v=u$ in \eqref{eq:second_order_frechet_derivative_for_KLD_mu_u_bar_mu_0_covariation_form}, and using \eqref{eq:special_choice_u} and \eqref{eq:distribution_first_exit_location}, we obtain
 \begin{align*}
  \bb{E}^{u}\left[(M^u_\tau)^2\left((M^u_\tau)^2+4M^u_\tau+2\right)\right]
  =&\bb{E}_{\bb{P}}\left[B_\tau^4+4 B_\tau^3+2B_\tau^2\right]
  \\
  =&\tfrac{b}{2+b}(16-32+8)+\tfrac{2}{2+b}(b^4+4b^3+2b^2)
  \\
  =&\tfrac{2}{2+b}\left(b^4+4b^3+2b^2-4b\right).
 \end{align*}
 Since the polynomial inside the parentheses factorises as $b(b+2)(b^2+2b-2)$, it follows that
 \begin{equation*}
   \bb{E}^{u}\left[(M^u_\tau)^2\left((M^u_\tau)^2+4M^u_\tau+2\right)\right]=2b(b^2+2b-2)=:q(b).
 \end{equation*}
The roots of the polynomial $q(\cdot)$ are $b=-1-\sqrt{3}$, and $b=0$, $b=\sqrt{3}-1$. Over the interval $(-1-\sqrt{3},0)$, $q(\cdot)$ is strictly positive, and over the interval $(0,\sqrt{3}-1)$, $q(\cdot)$ is strictly negative. The local minimum of $q(\cdot)$ over $(0,\sqrt{3}-1)$ is attained at $b^\ast:=\tfrac{1}{3}(\sqrt{7}-1)$, and $q(b^\ast)\approx -1.26$. This proves that whenever the right endpoint $b$ in the definition of the first exit time belongs to the interval $(0,\sqrt{3}-1)$, then the left-hand side of \eqref{eq:second_order_frechet_derivative_for_KLD_mu_u_bar_mu_0_covariation_form} is strictly negative. By statement (a) of \cref{lem:second_derivative_test_of_convexity}, it follows that the map $u\mapsto \bb{E}^{u}[(M^u_\tau)^2]$ is nonconvex at $u=1$. 
\end{example}

\begin{remark}
 In \cref{exa:change_of_drift_to_KLD_map_nonconvex}, the bound in \eqref{eq:second_order_frechet_derivative_for_KLD_mu_u_bar_mu_0_covariation_form_lower_bound} is not attained, even when we choose the optimal $b^\ast$. This is to be expected, since the bound in \eqref{eq:second_order_frechet_derivative_for_KLD_mu_u_bar_mu_0_covariation_form_lower_bound} is attained if and only if $\mu^{u}(p(M^u_\tau)=-2)=1$. One can ensure the latter statement holds true by defining $\tau$ as the first passage time to the minimiser $x^\ast=-2$ of the polynomial $p(\cdot)$ in \eqref{eq:second_order_frechet_derivative_for_KLD_mu_u_bar_mu_0_covariation_form}. However, since the first passage time of Brownian motion is not integrable, \cref{asmp:ball_of_admissible_changes_of_drift} will not hold in this case.
\end{remark}

\begin{remark}
 \label{rem:interpretation_of_exa_change_of_drift_to_KLD_map_nonconvex}
Given a measurable space $(\Omega',\mathcal{F}')$ and the set $\mathcal{P}(\Omega')$ of probability measures on $(\Omega',\mathcal{F}')$, it is known that the map $\mathcal{P}(\Omega')\times\mathcal{P}(\Omega')\ni (\mu,\nu)\mapsto \KLD{\mu}{\nu}$ is convex, see e.g. \cite[Theorem 2.7.2]{CoverThomas_1991}. The significance of \cref{exa:change_of_drift_to_KLD_map_nonconvex} with respect to the convexity of the map $(\mu,\nu)\mapsto \KLD{\mu}{\nu}$ is that the \emph{composition} of the map $u\mapsto (\mu^{u},\mu^{0})\in\mathcal{P}(\mathbf{W}\times\mathbf{W})\times\mathcal{P}(\mathbf{W}\times\mathbf{W})$ with the convex map $(\mu^{u},\mu^{0})\mapsto \KLD{\mu^{u}}{\mu^{0}}$ is not in general convex.  
\end{remark}

Let $\lambda>0$, and define the map 
 \begin{equation}
 \label{eq:regularised_KLD_functional_Phi}
\Phi:U\to\bb{R},\quad  u\mapsto \bb{E}^{u}[\phi_\tau +\lambda (M^u_\tau)^2].
 \end{equation}
 The following theorem characterises the smoothness and convexity properties of $\Phi(\cdot)$.
\begin{theorem}
 \label{thm:strict_convexity_of_regularised_KLD}
 Let $\Phi:U\to\bb{R}$ be defined by \eqref{eq:regularised_KLD_functional_Phi}. 
 \begin{enumerate}
  \item[(i)] If \cref{asmp:uniform_lower_bound_on_smallest_singular_value_of_diffusion} and \cref{asmp:ball_of_admissible_changes_of_drift} hold, and if for every $u\in U$ it holds that $\phi_\tau\in L^2(\mu^{u})$, then $\Phi$ is twice Fr\'{e}chet differentiable on $U$. The first Fr\'{e}chet derivative at $u\in U$ is given by
  \begin{align}
  \label{eq:first_Frechet_derivative_regularised_KLD_functional_Phi}
  V\ni v\mapsto  D\Phi(u)(v)=&\bb{E}^{u}\left[ \left(\phi_\tau+\lambda\left((M^u_\tau)^2+ 2M^u_\tau\right)\right) M^v_\tau \right]
 \end{align}
  If in addition $\phi_\tau$ is not $\mu^{u}$-almost surely constant, then the second Fr\'{e}chet derivative at $u\in U$ is given by
  \begin{equation}
  \label{eq:second_Frechet_derivative_regularised_KLD_functional_Phi}
  \begin{aligned}
   V\times V \ni (v,w)\mapsto D^2\Phi(u)(v,w)= &\bb{E}^{u}\left[ \left(\phi_\tau+\lambda\left((M^u_\tau)^2+4M^u_\tau+2\right)\right)M^w_\tau M^v_\tau \right]
   \end{aligned}
  \end{equation}
  \item[(ii)] If in addition \cref{asmp:uniform_ellipticity} and \cref{asmp:uniform_lower_bound_on_path_functional} hold, then $\Phi$ is also strictly convex, and there exists at most one $u^\ast\in U$ such that the first Fr\'{e}chet derivative of $\Phi$ vanishes at $u^\ast$. If $u^\ast\in U$, then $u^\ast$ is the unique element of $U$ such that
  \begin{equation*}
   \forall v\in V,\quad   \textup{Cov}^{u^\ast}\left[\phi_\tau+\lambda\left(\left(M^{u^\ast}_\tau)^2+2M^{u^\ast}_\tau\right)\right),M^v_\tau\right]=0.
  \end{equation*}  
 \end{enumerate}
\end{theorem}
\begin{proof}[Proof of \cref{thm:strict_convexity_of_regularised_KLD}]
For statement (i), \eqref{eq:first_Frechet_derivative_regularised_KLD_functional_Phi} follows from the fact that the Fr\'{e}chet derivative is a linear operator and from using \cref{lem:frechet_derivative_for_functional_not_depending_on_perturbation} and \cref{lem:frechet_derivative_for_KLD_mu_u_bar_mu_0_clmg_form}. Similarly, \eqref{eq:second_Frechet_derivative_regularised_KLD_functional_Phi} follows from using  \cref{lem:mixed_second_order_frechet_derivative_for_functional_not_depending_on_perturbation} and \cref{lem:mixed_second_order_frechet_derivative_for_KLD_mu_u_bar_mu_0_clmg_form}. 

To prove statement (ii), assume first that the lower bound $c$ of $\phi$ as stated in \cref{asmp:uniform_lower_bound_on_path_functional} satisfies $c-2\lambda>0$. Then by \eqref{eq:second_Frechet_derivative_regularised_KLD_functional_Phi} and \eqref{eq:second_order_frechet_derivative_for_KLD_mu_u_bar_mu_0_covariation_form_lower_bound}, the second Fr\'{e}chet derivative of $\Phi(\cdot)$ at $u$ along the directions $(v,v)\in V\times V$ satisfies
 \begin{equation*}
D^2 \Phi(u)(v,v)=\bb{E}^{u}\left[\left(\phi_\tau  +\lambda\left(((M^{u}_\tau)^2+4M^{u}_{\tau}+2\right)\right)(M^{v}_\tau)^2\right]\geq  (c-2\lambda)\bb{E}^{u}[(M^v_\tau)^2].
 \end{equation*}
 Using \cref{asmp:uniform_lower_bound_on_path_functional} with constant $c$ such that $c-2\lambda>0$, it follows from the inequality that for all $v\neq 0$, $D^2\Phi(u)(v,v)>0$. Hence, by part (b) of \cref{lem:second_derivative_test_of_convexity}, $\Phi:U\to\bb{R}$ is strictly convex.

 Now suppose that the lower bound $c$ of $\phi$ in \cref{asmp:uniform_lower_bound_on_path_functional} satisfies $c-2\lambda \leq 0$. Then $\hat{\phi}:=\phi-(c-3\lambda)$ satisfies \cref{asmp:uniform_lower_bound_on_path_functional} with constant lower bound $\lambda>0$. Next, define $\widehat{\Phi}:U\to\bb{R}$ by $\widehat{\Phi}(u):=\bb{E}^{u}[\hat{\phi}_\tau+\lambda (M^u_\tau)^2]=\Phi(u)-(c-3\lambda)$. Since $\hat{\phi}$ satisfies \cref{asmp:uniform_lower_bound_on_path_functional} with strictly positive constant lower bound, we may apply the argument in the preceding paragraph to prove that $\widehat{\Phi}(\cdot)$ is strictly convex. Since $\Phi(\cdot)$ is the translation of $\widehat{\Phi}(\cdot)$ by a constant, it follows from  \cref{lem:translation_invariance_convex_maps} that $\Phi(\cdot)$ is strictly convex. The statement involving the covariance follows from using \eqref{eq:first_Frechet_derivative_regularised_KLD_functional_Phi}, the definition of the covariance, and the fact that $\bb{E}^{u}[M^v_\tau]=0$ for every $v\in V$, by \eqref{eq:clmg}.
\end{proof}

\section{Application to first exit stochastic optimal control problems}
\label{sec:application_stochastic_optimal_control}

In this section, we consider the case of stochastic optimal control problems defined by the first exit of a diffusion process from an open bounded domain $D\subset \bb{R}^{d}$. Let $g:\bb{R}_{\geq 0}\times D\to\bb{R}^{d}$ and $f:\bb{R}_{\geq 0}\times D\to\bb{R}^{d\times d}$ be such that there exists a weak solution to \eqref{eq:sde} that is unique in law, as in \cref{asmp:exists_soln_to_sde_unique_in_law}. Define the stopping time $\tau$ to be the first exit time from $D$,
\begin{equation}
 \label{eq:first_exit_time}
 \tau(w):=\inf\left\{t\geq 0\ \middle\vert\ w_t\notin D\right\}.
\end{equation}
Let $k_\run:\bb{R}_{\geq 0}\times D\to \bb{R}$ and $k_\term:\bb{R}_{\geq 0}\times \partial D\to\bb{R}$. We shall refer to these as the `running cost' and `terminal cost' function respectively. Let $\phi:\bb{R}_{\geq 0}\times \mathbf{W}\to\bb{R}$ be a predictable function defined by
\begin{equation}
 \label{eq:path_functional}
 \phi_t(w):=\int_0^{t}k_\run(s,w_s)\ud s +k_\term(t,w_t), \quad (t,w)\in\bb{R}_{\geq 0}\times\mathbf{W}.
\end{equation}
Let $\lambda>0$. The first exit stochastic optimal control problem is defined by
\begin{equation}
\begin{aligned}
\label{eq:first_exit_control_problem}
\min_{u\in U}\Phi(u;\lambda)= &\bb{E}\left[\phi_{\tau(X^{u})}(X^{u})+\lambda \cdot\tfrac{1}{2}\Ang{M^u,M^u}_{\tau(X^{u})}(X^u)\right]
\\
\text{ such that }\ud X^{u}_t=&(u+g)(t,X^{u}_t)\ud t+f(t,X^{u}_t)\ud B_t,\quad X^{u}_0\in D.
\end{aligned}
\end{equation}
In \eqref{eq:first_exit_control_problem}, the `control' function $u$ is the change of drift. The space $U$ of `admissible controls' is a subset of $L^\infty(\bb{R}_{\geq 0}\times D;\bb{R}^d)$ that satisfies \cref{asmp:ball_of_admissible_changes_of_drift}. 

 The objective $\Phi(u;\lambda)$ is the same as \eqref{eq:regularised_KLD_functional_Phi}, up to the scaling by $\tfrac{1}{2}>0$. Using the formula $\tfrac{1}{2}\bb{E}^{u}[\Ang{M^u,M^u}_\tau]=\KLD{\mu^{u}}{\mu^{0}}$ in \eqref{eq:KLD}, it follows that the second term in the objective functional of \eqref{eq:first_exit_control_problem} plays the role of an entropic regularisation term, and the parameter $\lambda>0$ quantifies the strength of the regularisation: for larger values of $\lambda$, the relative entropy plays a more important role in the objective functional $\Phi(\cdot;\lambda)$ than $\phi_\tau$. \Cref{exa:change_of_drift_to_KLD_map_nonconvex} shows that the entropic regularisation term is not in general a convex function of $u$. However, we have the following result.
\begin{corollary}
 \label{cor:strict_convexity_of_first_exit_control_objective_functional}
 Suppose that \cref{asmp:ball_of_admissible_changes_of_drift} and \cref{asmp:uniform_ellipticity} hold. If $k_\run(\cdot,\cdot)$ is nonnegative and $k_\term(\cdot,\cdot)$ is bounded from below on $\bb{R}_{\geq 0}\times D$, then for every $\lambda>0$, the objective functional $\Phi(\cdot~;\lambda):U\to\bb{R}$ of the stochastic optimal control problem \eqref{eq:first_exit_control_problem} is twice Fr\'{e}chet differentiable and strictly convex. Hence, \eqref{eq:first_exit_control_problem} has a unique solution.
\end{corollary}
\begin{proof}
 Since $k_\run(\cdot,\cdot)$ is nonnegative, the predictable function $\phi(\cdot,\cdot)$ defined in \eqref{eq:path_functional} satisfies\cref{asmp:uniform_lower_bound_on_path_functional} with lower bound given by $c=\inf\{k_\term(t,x)\ :\ (t,x)\in\bb{R}_{\geq 0}\times D\}$. Thus, the result follows from \cref{thm:strict_convexity_of_regularised_KLD}.
\end{proof}

We recall some main ideas from \cite{HartmannSchuette_2012}. First exit control problems of the form \eqref{eq:first_exit_control_problem} arise in the following way. Given $\tau(\cdot)$ and $\phi(\cdot)$ as defined in \eqref{eq:first_exit_time} and \eqref{eq:path_functional} respectively, define $F:\bb{R}_{>0}\to\bb{R}$ according to
\begin{equation*}
   F(\lambda):=-\lambda\log \bb{E}^{0}\left[\exp(-\lambda^{-1}\phi_\tau) \right].
\end{equation*}
The cumulant generating function $F(\cdot)$ contains information about the concentration of the distribution of $\phi$ with respect to $\mu^{0}$ around its mean. In the context of statistical physics, if $\phi_{\tau(X^u)}(X^{u})$ is given the interpretation of the work done by the controlled process $X^{u}$ up to the first exit time of $X^u$ from the domain $D$, then the formula for $F(\lambda)$ is similar to the formula for the associated free energy. Recall that $\mu^{0}$ denotes the law of the solution to the SDE \eqref{eq:sde} without control, i.e. with $u\equiv 0$. If $u$ is such that $\mu^{0}$ is absolutely continuous with respect to the law $\mu^{u}$ of the solution $(X^u,B)$ to \eqref{eq:sde_perturbed}, then
\begin{align}
 F(\lambda)&=-\lambda\log\bb{E}[\exp(-\lambda^{-1}\phi_{\tau})\mathcal{E}(M^{-u})_{\tau}] 
\nonumber \\
 &\leq \bb{E}^{u}\left[\phi_{\tau}+\lambda M^{u}_\tau+ \lambda\cdot\tfrac{1}{2}\Ang{M^{-u},M^{-u}}_\tau\right]=\Phi(u;\lambda),
 \label{eq:derivation_of_objective_functional}
\end{align}
for $\Phi(u;\lambda)$ as defined in \eqref{eq:first_exit_control_problem}. Above, we used the fact that $M^{u}$ is a continuous local martingale with $M^{u}_0=0$ by \eqref{eq:clmg}. The inequality above expresses the cumulant generating function as the value function of the optimal control problem \eqref{eq:first_exit_control_problem}. There exist sufficient conditions involving the domain $D$, the coefficients $f,g$ in \eqref{eq:sde}, and the functions $\kappa_\run$ and $\kappa_\term$, such that the value function is the unique classical solution to the Hamilton-Jacobi-Bellman equation of the stochastic optimal control problem. However, this is a vast topic that we shall not investigate in this paper.

\subsection{Convergence of gradient descent-based method}
\label{ssec:convergence_of_gradient_descent_algorithm}

In \cite{HartmannSchuette_2012}, the authors proposed the following numerical method for solving \eqref{eq:first_exit_control_problem} using gradient descent, in the special case that the SDE in the constraint is given by 
\begin{equation}
\label{eq:sde_perturbed_special_case}
 \ud X^{u}_t=\left(u-\nabla V\right)(X^{u}_t)\ud t+\sqrt{2\epsilon}\ud B_t,\quad X^u_0=x\in D,
\end{equation}
where $V:D\to\bb{R}$ is a $C^1$ potential function and $\epsilon>0$. The SDE above is discretised using the Euler-Maruyama method. A finite collection $(b_k(\cdot))_{k=1}^{n}$ of linearly independent elements of the space $U$ of admissible controls is chosen, where for each $k=1,\ldots,n$, $b_k:D\to\bb{R}^{d}$. The stochastic optimal control problem \eqref{eq:first_exit_control_problem} is approximated using the finite-dimensional optimisation problem
\begin{equation}
\label{eq:finite_dimensional_optimisation_problem}
 \min_{a\in\bb{R}^{n}}\Phi (u^a;\lambda),\quad u^a(\cdot):=\sum_{i=1}^{n}a_ib_i(\cdot),
\end{equation}
subject to the constraint \eqref{eq:sde_perturbed_special_case}. An initial coefficient vector $a^{(0)}\in\bb{R}^{n}$ is chosen, and the coefficient vector is updated sequentially according to
\begin{equation}
\label{eq:gradient_descent_step}
 a^{(j)}=a^{(j-1)}+h^{(j-1)}\nabla_a \Phi \left(u^a;\lambda\right),\quad j=1,\ldots, J,
\end{equation}
where $(h^{(j)})_{j=0}^{J-1}$ are a sequence of strictly positive step sizes and $\nabla_a$ denotes the gradient with respect to the coefficient vector $a$. 

The following result is a corollary of \cref{thm:strict_convexity_of_regularised_KLD}(i). Below, $D_{a_k}$ denotes the partial derivative with respect to $a_k$, and $D_{a_k a_\ell}$ denotes the second partial derivative with respect to $a_k$ and $a_\ell$.
\begin{corollary}
 \label{cor:partial_derivatives_finite_dimensional_objective_function}
 Assume that  \cref{asmp:uniform_lower_bound_on_smallest_singular_value_of_diffusion} and \cref{asmp:ball_of_admissible_changes_of_drift} hold. Then the first- and second-order partial derivatives of the function $a\mapsto \Phi (u^a;\lambda)$ are given by
 \begin{align}
  D_{a_k} \Phi \left(u^a;\lambda\right)&= \bb{E}^{u^a}\left[ \left(\phi_\tau+\lambda\left(\tfrac{1}{2}\left(M^{u^a}_\tau\right)^2+M^{u^a}_{\tau}\right)\right)M^{b_k}_\tau\right]
\label{eq:continuous_version_partial_derivative_finite_dim_objective}
\\
  D_{a_k a_\ell} \Phi \left(u^a;\lambda\right)&=\bb{E}^{u^a}\left[\left(\phi_\tau+\lambda\left(\tfrac{1}{2}\left(M^{u^a}_\tau\right)^2+2M^{u^a}_{\tau}+1\right)\right) M^{b_k}_\tau M^{b_\ell}_\tau\right]
  \label{eq:continuous_version_second_partial_derivative_finite_dim_objective}
\end{align}
\end{corollary}
\begin{proof}
 Let $(e_i)_{i=1}^{n}$ be the standard orthonormal basis of $\bb{R}^{n}$, and let $k\in\{1,\ldots,n\}$ be arbitrary. The definition of $u^a(\cdot)$ in  \eqref{eq:finite_dimensional_optimisation_problem} implies that $u^{a+he_k}(\cdot)=u^{a}(\cdot)+hb_k(\cdot)$. Hence, for the objective functional $\Phi(\cdot~;\lambda)$ defined in \eqref{eq:first_exit_control_problem},
\begin{equation*}
D_{a_k} \Phi \left(u^a;\lambda\right)=\lim_{h\to 0}\frac{ \Phi (u^{a+he_k};\lambda)-\Phi (u^a;\lambda)}{h}=\lim_{h\to 0}\frac{\Phi (u^a+h b_k;\lambda)-\Phi (u^a)}{h}.
\end{equation*}
The last limit is the Gateaux derivative of the objective functional $\Phi (\cdot~;\lambda)$ evaluated at $u^a$ in the direction $b_k$. Since the Fr\'{e}chet derivative exists, it coincides with the Gateaux derivative. Thus, \eqref{eq:continuous_version_partial_derivative_finite_dim_objective} and \eqref{eq:continuous_version_second_partial_derivative_finite_dim_objective} follow from \eqref{eq:first_Frechet_derivative_regularised_KLD_functional_Phi} and \eqref{eq:second_Frechet_derivative_regularised_KLD_functional_Phi}  in \cref{thm:strict_convexity_of_regularised_KLD}(i), after replacing $u$ with $u^a$, $\lambda$ with $\lambda\cdot\tfrac{1}{2}$, $v$ with $b_k$, and $w$ with $b_\ell$.
\end{proof}

We make some observations that illustrate the significance of the formula above -- and by extension, the results in \cref{sec:convexity} -- with respect to some open problems that were raised in \cite{HartmannSchuette_2012}.

\paragraph{Exact vs. inexact gradients.} In \cite[Eq. (A.4)]{HartmannSchuette_2012}, an expression is given for the $j$-th partial derivative of the finite-dimensional objective function $a\mapsto \Phi (u^a;\lambda)$, under the assumption that the stopping time $\tau$ in \eqref{eq:path_functional} is deterministic and finite. The value of this expression is the same as that of the Euler-Maruyama discretisation of \eqref{eq:continuous_version_partial_derivative_finite_dim_objective}, except that the expression \cite[Eq. (A.4)]{HartmannSchuette_2012} is given for a deterministic, finite stopping time. The authors propose an inexact gradient to take into account the random stopping time case, but conclude that ``it is unclear how [the partial derivatives of $\tau(X^{u})$] can be handled numerically efficiently''. The formula \eqref{eq:continuous_version_partial_derivative_finite_dim_objective} resolves this problem by showing that the expression given in \cite[Eq. (A.4)]{HartmannSchuette_2012} extends in the natural way to random stopping times, i.e. by replacing the discrete deterministic stopping time with the discrete random stopping time.

\paragraph{Convergence of the gradient descent method.} In \cite{HartmannSchuette_2012}, there is no convergence analysis of the gradient descent-based method. If the assumptions of \cref{cor:strict_convexity_of_first_exit_control_objective_functional} are satisfied, then any gradient-descent based method will converge to a unique minimiser of the finite-dimensional objective function $a\mapsto \Phi (u^a;\lambda)$ in \eqref{eq:finite_dimensional_optimisation_problem}, provided that expected values are evaluated exactly and the SDE is evaluated exactly. This is because the finite-dimensional objective function inherits the strict convexity property of the infinite-dimensional objective functional $\Phi (\cdot~;\lambda)$. In practice, the convergence will be masked by the statistical error due to the approximation of expected values, by the error due to the discretisation of the SDE, and by the error due to the gradient descent step in \eqref{eq:gradient_descent_step}. 

\paragraph{Deterministic limit of the finite-dimensional problem.} The method considered in \cite{HartmannSchuette_2012} is developed by first specifying a finite basis $(b_k(\cdot))_{k=1}^{n}$ of functions. Now suppose one is given a basis $(b_k)_{k\in\bb{N}}$ for the infinite-dimensional set $U$ of admissible controls, and computes for each $n\in\bb{N}$ the solution of the finite-dimensional problem \eqref{eq:finite_dimensional_optimisation_problem}. It is natural to ask whether the sequence of solutions converges to the solution of the infinite-dimensional problem \eqref{eq:first_exit_control_problem} as $n\to\infty$. The preceding analysis answers this question affirmatively, since we obtained the finite-dimensional problem by restricting the infinite-dimensional problem -- which has a unique solution, by \cref{cor:strict_convexity_of_first_exit_control_objective_functional} -- to the span of the first $n$ basis functions.

\subsection{Extensions to the gradient descent-based method}

\paragraph{Optimisation methods that use second derivatives.} In \cite{HartmannSchuette_2012}, the authors propose the gradient descent method \eqref{eq:gradient_descent_step} to solve the stochastic optimal control problem \eqref{eq:finite_dimensional_optimisation_problem}. One can also use methods that use second derivatives of the objective function, such as Newton's method for solving the finite-dimensional optimisation problem \eqref{eq:finite_dimensional_optimisation_problem}: for $j=1,\ldots,J$,
\begin{equation*}
 a^{(j)}=a^{(j-1)}+J(a^{(j-1)})^{-1} \nabla_a\Phi (u^a;\lambda)
\end{equation*}
where $a\mapsto J(a)\in\bb{R}^{n\times n}$ is the Hessian matrix of $a\mapsto \Phi (u^a;\lambda)$, with entries given by \eqref{eq:continuous_version_second_partial_derivative_finite_dim_objective}. The new vector $a^{(j)}$ may be obtained without inverting the matrix $J(a^{(j-1)})$, by solving 
\begin{equation*}
 J(a^{(j-1)})(a^{(j)}-a^{(j-1)})=\nabla_a\Phi (u^a;\lambda)
\end{equation*}
for $(a^{(j)}-a^{(j-1)})$ and using the value of the current vector $a^{(j-1)}$. The formulas \eqref{eq:continuous_version_partial_derivative_finite_dim_objective} and \eqref{eq:continuous_version_second_partial_derivative_finite_dim_objective} show that one can compute both the gradient and the Hessian of $a\mapsto \Phi (u^a;\lambda)$ using only $\phi_\tau$ and $(M^{b_k}_\tau)_{k=1}^{n}$, since $M^{u^a}_\tau=\sum_{k=1}^{n}a_k M^{b_k}_\tau$ by the linearity of the map $v\mapsto M^v_\tau$ in \eqref{eq:clmg} and the fact that $u^a(\cdot)$ is defined in \eqref{eq:finite_dimensional_optimisation_problem} as a linear combination of the $(b_k(\cdot))_{k=1}^{n}$. Thus, the Newton method can be implemented using the same path functionals that are needed to implement the gradient descent method. The same statement is valid for any numerical optimisation method that uses second-derivative information of the function $a\mapsto\Phi(u^a;\lambda)$.

Optimisation methods that use second derivative information about the objective may be useful in cases where the first exit of $X^u$ from $D$ is a rare event, i.e. when the first exit time $\tau(X^u)$ takes large values with high probability. In such cases, for any simulated trajectory of $X^u$, the discretisation of $\tau(X^u)$ will be given by $N_{\tau(\widetilde{X^u})}\Delta t$, where $\Delta t$ denotes the integration time step, $\widetilde{X^u}$ denotes the discrete approximation of $X^u$, and $N_{\tau(\widetilde{X^u})}\in\bb{N}$ is the discrete approximation of the exit time. In particular, given that the first exit of $X^u$ from $D$ is a rare event, $N_\tau(w)$ will be larger than the dimension $n$ of the finite-dimensional space in \eqref{eq:finite_dimensional_optimisation_problem}, with high probability. Since the cost of computing the value of a single path functional such as $M^{b_k}_{\tau}$ for a single realisation of $(X^{u}_t)_{0\leq t\leq \tau(X^{u})}$ grows linearly with $N_{\tau(\widetilde{X^u})}$, the cost of the additional operations of the optimisation method will be small in comparison to the cost of computing the path functionals themselves, whenever $N_{\tau(\widetilde{X^u})}$ is large. 

\paragraph{Control variates.} In practical situations, the expected values in $\Phi(u^a;\lambda)$ and its derivatives cannot be evaluated exactly, and must be approximated using a Monte Carlo method. It is then of interest to reduce the sample variance of the corresponding Monte Carlo estimates. 

In the present context, the method of control variates may be an effective method for variance reduction. Suppose one wishes to estimate an expected value $\bb{E}[\varphi(Y)]$, where $Y$ is a random variable and $\varphi(\cdot)$ is a $\bb{R}$-valued function. In addition, one knows the expected value $\bb{E}[\kappa(Y)]$ for the same random variable $Y$ but a different $\bb{R}$-valued function $\kappa(\cdot)$. The random variable $\kappa(Y)$ is the `control variate' or `control variable'. Then the control variate-based estimator $\varphi(Y)+\beta(\bb{E}[\kappa(Y)]-\kappa(Y))$ is an unbiased estimator of $\bb{E}[\varphi(Y)]$, with variance $\textup{Var}(\varphi(Y))+\beta^2 \textup{Var}(\kappa(Y))-2\beta \textup{Cov}(\varphi(Y),\kappa(Y))$. By differentiating with respect to $\beta$, one can show that the optimal choice of $\beta$ is
\begin{equation*}
 \beta^\ast=\textup{Cov}(\varphi(Y),\kappa(Y))/\textup{Var}(\kappa(Y)).
\end{equation*}
The variance of the corresponding estimator is $(1-\text{Corr}(\varphi(Y),\kappa(Y))^2)\text{Var}(\varphi(Y))$. 

It is known that control variates can decrease the computational efficiency of a Monte Carlo method, if the control variate is not sufficiently correlated with the random variable of interest \cite[Section 4.4]{Lemieux2009}. Recall the following quantitative heuristic for determining whether a control variate will increase the computational efficiency \cite[Section 8.9]{Owen2013}. Denote the average cost of generating a sample of $Y$ and $\varphi(Y)$ by $c_\varphi$, and denote the average cost of evaluating $\beta(\bb{E}[\kappa(Y)]-\kappa(Y))$ \emph{given} a sample value of $Y$ by $c_\kappa$. That is, $c_\kappa$ includes the cost of evaluating $\kappa(Y)$, but not of sampling $Y$. If the cost of evaluating the Monte Carlo approximation of $\beta^\ast$ is small relative to $c_\kappa$ and $c_\varphi$, and $\abs{\text{Corr}(\varphi(Y),\kappa(Y))}>\sqrt{ c_\kappa/(c_\varphi+c_\kappa)}$, then the method using the control variate $\kappa(Y)$ yields a more efficient Monte Carlo method.

Now we apply the preceding discussion to the stochastic optimal control problem \eqref{eq:finite_dimensional_optimisation_problem}. The random path $X^{u^a}$ and the path functional $(\phi_\tau+\lambda\cdot\tfrac{1}{2}(M^{u^a}_\tau)^2)(\cdot)$ correspond to the random variable $Y$ and the function $\varphi(\cdot)$. From \eqref{eq:derivation_of_objective_functional} it is natural to use $(M^{u^a}_\tau)(\cdot)$ as $\kappa(\cdot)$. In \cite[Section 4.2]{Lie2015}, this choice of control variate was described, but the computational efficiency of the control variate method was not discussed. 

The advantage to using $(M^{u^a}_\tau)(X^{u^a})$ as the control variate is that there is no additional cost of computing it, because we need to compute $(M^{u^a}_\tau)(X^{u^a})$ to compute $(\phi_\tau+\lambda\cdot\tfrac{1}{2}(M^{u^a}_\tau)^2)(X^{u^a})$. In addition, $(M^{u^a}_\tau)(X^{u^a})$ is correlated with $(\phi_\tau+\lambda\cdot\tfrac{1}{2}(M^{u^a}_\tau)^2)(X^{u^a})$. Given the cost argument in the preceding paragraph, we expect that using $(M^{u^a}_\tau)(X^{u^a})$ as a control variate will tend to increase the efficiency of the Monte Carlo method.

For any nonzero $z\in\bb{R}^{n}$, $(M^{u^z}_\tau )(X^{u^a})$ is a candidate for a control variate, since the expectation of this random variable is zero, by the martingale property and \eqref{eq:clmg}. The additional cost of computing the control variate is at most $2n-1$ floating-point operations, since these are the number of floating-point operations needed to compute $(M^{u^z}_\tau )(X^{u^a})$ from $((M^{b_1}_\tau )(X^{u^a}))_{k=1}^{n}$ and the vector $z\in\bb{R}^{n}$. Relative to computing a single realisation of $(X^{u^a}_t)_{0\leq t\leq \tau(X^{u^a})}$, $2n-1$ floating-point operations is small, and so one expects that the lower bound for the absolute value of the correlation will be small. Thus, to obtain an increase in efficiency of the Monte Carlo method, one needs to check that the correlation between $(M^{u^z}_\tau )(X^{u^a})$ and $(\phi_\tau+\lambda\cdot\tfrac{1}{2}(M^{u^a}_\tau)^2)(X^{u^a})$ is sufficiently large.

The control variate approach may be applied to the Monte Carlo estimation of the entries $(D_{a_k}\Phi(u^a;\lambda))_{k=1}^{n}$ of the gradient $\nabla \Phi(u^a;\lambda)$ as well as the entries $(D_{a_k a_\ell}\Phi(u^a;\lambda))_{k,\ell=1}^{n}$ of the Hessian. For $k,\ell\in\{1,\ldots,n\}$, candidate control variates for estimating $D_{a_k}\Phi(u^a;\lambda)$ and $D_{a_ka_\ell}\Phi(u^a;\lambda)$ are $M^{u^a}_\tau(X^{u^a})$ and $M^{b_k}_{\tau}(X^{u^a})$. If one redefines the function $\varphi(\cdot)$ to be $(\phi_\tau+\lambda(\cdot\tfrac{1}{2}(M^{u^a}_\tau)^2+M^{u^a}_\tau)(\cdot)$ and uses $M^{u^z}_\tau(X^{u^a})$ as the control variate, then the optimal $z^\ast\in\bb{R}^{n}$ which achieves the largest reduction in variance is obtained by solving
\begin{equation*}
\textup{Var}(M^{b_1}_\tau(X^{u^a}),\ldots, M^{b_n}_\tau(X^{u^a})) x=\nabla_a\Phi(u^a;\lambda)
\end{equation*}
for $x$. Above, $\textup{Var}(M^{b_1}_\tau(X^{u^a}),\ldots, M^{b_n}_\tau(X^{u^a}))\in\bb{R}^{n\times n}$ is a matrix whose $(k,\ell)$-th entry is given by $\textup{Cov}(M^{b_k}_\tau(X^{u^a}),M^{b_\ell}_\tau(X^{u^a}))$ for $k,\ell\in\{1,\ldots,n\}$. The entries of $\nabla_a\Phi(u^a;\lambda)$ are given by \eqref{eq:continuous_version_partial_derivative_finite_dim_objective}. The equation above expresses a relationship between the optimal control variate estimator for $\bb{E}[(\phi_\tau+\lambda(\cdot\tfrac{1}{2}(M^{u^a}_\tau)^2+M^{u^a}_\tau)(X^{u^a})]$ -- which has the same value as $\Phi(u^a;\lambda)$ -- and the gradient of the map $a\mapsto \Phi(u^a;\lambda)$. However, it is not possible to obtain a similar relationship between a control variate-based estimator for $D_{a_k}\Phi(u^a;\lambda)$ and the gradient of the map $a\mapsto D_{a_k}\Phi(u^a;\lambda)$. This is because of the change in the scaling of $M^{u^a}_\tau M^{b_k}_\tau$ from 1 to 2 when proceeding from \eqref{eq:continuous_version_partial_derivative_finite_dim_objective} to \eqref{eq:continuous_version_second_partial_derivative_finite_dim_objective}.

\section{Proofs}
\label{sec:proofs}

Below, $f$ and $g$ denote the diffusion and drift coefficients of \eqref{eq:sde}, $\tau$ and $\phi$ refer to a fixed stopping time and predictable function, and $V$ is the space of bounded functions \eqref{eq:V_vector_space_of_changes_of_drift}. 

Recall from \cref{lem:uniform_bound_on_quadratic_variation_time_t} that \cref{asmp:uniform_lower_bound_on_smallest_singular_value_of_diffusion} implies
\begin{equation*}
\Ang{M^v,M^w}_{\tau}\leq \alpha^{-2} \Vrt{v}_V\Vrt{w}_V\tau.
\end{equation*}

For the next result, we will need the following theorem, which we quote from \cite[Section 1.2, Theorem 1.5]{Kazamaki_1994}. 
\begin{theorem}
\label{thm:kazamakis_theorem}
	Let $M$ be a continuous local martingale on $(\Omega',(\mathcal{F}_t')_{t\geq 0},\bb{P}')$. Let $1<p,q<\infty$ satisfy $p^{-1}+q^{-1}=1$. If
	\begin{equation}\label{eq:kazamakis_condition_for_Lq_boundedness_of_exp_MG}
		\sup\left\{\bb{E}_{\bb{P}'}\left[\exp\left(\frac{1}{2}\frac{\sqrt{p}}{\sqrt{p}-1}M_T\right)\right]\ \middle\vert\ T\text{ a $\bb{P}'$-a.s. bounded stopping time }\right\}<\infty, 
	\end{equation}
	then the exponential martingale $\mathcal{E}(M)$ defined by \eqref{eq:exponential_martingale} is an $L^q(\bb{P}')$-bounded martingale.
\end{theorem}
We will use \cref{thm:kazamakis_theorem} to establish the following integrability result. 
\begin{corollary}
\label{cor:square_integrability_of_change_of_measure_at_stopping_time}
Suppose that \cref{asmp:uniform_lower_bound_on_smallest_singular_value_of_diffusion} holds with parameter $\alpha$. Let $u\in V$ and suppose there exists $\lambda_u>0$ such that $\bb{E}^{u}[\exp(\lambda_u \tau)]$ is finite. Let $1<p,q<\infty$ satisfy $p^{-1}+q^{-1}=1$. If $v\in V$ satisfies 
\begin{equation*}
\Vrt{v}_{V}^{2}\leq \lambda_u \left( \tfrac{1}{4}\left(\tfrac{\sqrt{p}}{\sqrt{p}-1}\right)^2\alpha^{-2} \right)^{-1}
\end{equation*}
then $\mathcal{E}(M^{v})_\tau\in L^q(\mu^{u})$.
\end{corollary}
The integrability result above is important for the main results of \cref{sec:derivatives}. It shows that as $\Vrt{v}_V$ decreases to zero, one can choose $p$ to be closer to 1, which means that $q$ may increase to $\infty$. Thus, the result implies that as $\Vrt{v}_V$ decreases to zero, the integrability $q$ of $\mathcal{E}(M^{v})_\tau$ increases. 

\begin{proof}[Proof of \cref{cor:square_integrability_of_change_of_measure_at_stopping_time}]
We shall use parts of the proof of \cite[Section 1.2, Theorem 1.5]{Kazamaki_1994}. 

Let $T$ be any $\mu^{u}$-a.s. bounded stopping time. By \eqref{eq:exponential_martingale},
 \begin{equation*}
  \exp\left(\tfrac{1}{2}M^v_{\tau\wedge T}\right)=\left(\mathcal{E}(M^v)_{\tau\wedge T}\right)^{1/2}\exp\left(\tfrac{1}{4}\Ang{M^v,M^v}_{\tau\wedge T}\right)
 \end{equation*}
where $s\wedge t:=\min\{s,t\}$. Thus, by the Cauchy-Schwarz inequality,
 \begin{align*}
  \bb{E}^{u}\left[ \exp\left(\tfrac{1}{2}M^v_{\tau\wedge T}\right)\right]\leq & \left(\bb{E}^{u}\left[\mathcal{E}(M^v)_{\tau\wedge T}\right]\bb{E}^{u}\left[\exp\left(\tfrac{1}{4}\Ang{M^v,M^v}_{\tau\wedge T}\right)\right]\right)^{1/2}
  \\
  =& \bb{E}^{u}\left[\exp\left(\tfrac{1}{4}\Ang{M^v,M^v}_{\tau\wedge T}\right)\right]^{1/2}.
 \end{align*}
If in the inequality above we replace the continuous local martingale $M^v$ with the continuous local martingale $\tfrac{\sqrt{p}}{\sqrt{p}-1}M^v$, then 
\begin{align*}
 \bb{E}^{u}\left[\exp\left(\tfrac{1}{2}\tfrac{\sqrt{p}}{\sqrt{p}-1}M^v_{\tau\wedge T}\right)\right] \leq& \bb{E}^{u}\left[\exp\left( \tfrac{1}{4}\left(\tfrac{\sqrt{p}}{\sqrt{p}-1}\right)^2\Ang{M^v,M^v}_{\tau\wedge T}\right)\right]^{1/2}
 \\
 \leq & \bb{E}^{u}\left[\exp\left( \tfrac{1}{4}\left(\tfrac{\sqrt{p}}{\sqrt{p}-1}\right)^2\alpha^{-2} \Vrt{v}_V^2 \tau\right)\right]^{1/2}
 \\
 \leq & \bb{E}^{u}\left[\exp\left(\lambda_u\tau\right)\right]^{1/2},
\end{align*}
where we used \cref{lem:uniform_bound_on_quadratic_variation_time_t} in the second inequality and the hypothesis on $\Vrt{v}_V$ in the third inequality. By the hypothesis on $\lambda_u$, the last quantity is finite. Thus, by \cref{thm:kazamakis_theorem}, the desired conclusion follows.
\end{proof}

\begin{lemma}
\label{lem:exponentially_integrable_stopping_time_on_mu_u}
 Let $u\in V$ and suppose there exists $\lambda_u>0$ such that $\bb{E}^{u}[\exp(\lambda_u \tau)]$ is finite. Let $1<p,q<\infty$ satisfy $p^{-1}+q^{-1}=1$. If $v\in V$ is such that $\mathcal{E}(M^v)_\tau\in L^q(\mu^{u})$, then
 \begin{equation*}
  \bb{E}^{u+v}\left[\exp\left(\tfrac{1}{p}\lambda_u\tau\right)\right]<\infty.
 \end{equation*}
\end{lemma}
\begin{proof}
 By applying H\"{o}lder's inequality and the change of measure formula \eqref{eq:change_of_measure},
 \begin{equation*}
  \bb{E}^{u+v}\left[\exp\left(\tfrac{1}{p}\lambda_u\tau\right)\right]
  \leq \bb{E}^{u}\left[\exp(\lambda_u\tau)\right]^{1/p}\bb{E}^{u}\left[\mathcal{E}(M^v)_\tau^q\right]^{1/q}.
 \end{equation*}
By the hypotheses, the right-hand side of the inequality is finite, as desired.
 \end{proof}
 
\begin{proof}[Proof of \cref{lem:existence_of_ball_of_admissible_changes_of_drift}]
By applying \cref{cor:square_integrability_of_change_of_measure_at_stopping_time} with $u=0$ and $p=q=2$, it follows that every $v\in U_0$ satisfies $\mathcal{E}(M^v)_\tau\in L^2(\mu^{0})$. Thus, by applying  \cref{lem:exponentially_integrable_stopping_time_on_mu_u} with $u=0$, $v\in U_0$, and $p=q=2$, it follows that $\bb{E}^{v}\left[\exp\left(\tfrac{1}{2}\lambda_0\tau\right)\right]$ is finite. Since $U$ in \cref{asmp:ball_of_admissible_changes_of_drift} is defined to be a convex set containing $0\in V$ such that for every $u\in U$, there exists $\lambda_u>0$ for which $\bb{E}^{u}[\exp(\lambda_u\tau)]$ is finite, it follows that $U_0\subset U$ must hold.

If $\tau$ is a deterministic fixed stopping time $T>0$, then $\bb{E}^{0}[\exp(\lambda_0 \tau)]=\exp(\lambda_0 T)$, so we may choose $\lambda_0>0$ to be arbitrarily large. Thus, $U$ contains every $V$-norm ball centred at the origin $0$ with positive radius. This implies that $U=V$.
\end{proof}
 
In \cref{lem:convergence_lemmas_for_change_of_measure}, we assume that  \cref{asmp:uniform_lower_bound_on_smallest_singular_value_of_diffusion} and \cref{asmp:ball_of_admissible_changes_of_drift} hold, and aim to show that for every $u\in U$ and for every $s\geq 1$,
\begin{align*}
\lim_{\Vrt{v}_V\to 0}\frac{ \Vrt{ \mathcal{E}(M^v)_\tau-1-M^v_\tau}_{L^s(\mu^{u})}}{\Vrt{v}_V}=&  0.
 \\
\lim_{\Vrt{v}_V\to 0}\frac{ \Vrt{ \mathcal{E}(M^v)_\tau-1}_{L^s(\mu^{u})}}{\Vrt{v}_V}= & 0.
\end{align*}

The proof of \cref{lem:convergence_lemmas_for_change_of_measure} given below uses a similar strategy as the proof of \cite[Lemma 2.8]{Koltai_2019}, which considered the special case where the stopping time $\tau$ is a deterministic, finite number $T>0$.
\begin{proof}[Proof of \cref{lem:convergence_lemmas_for_change_of_measure}]
Let $s \geq 1$. Recall \eqref{eq:convergence_lemma_for_change_of_measure_minus_one_minus_CLMG}:
\begin{equation}
\label{eq:convergence_lemma_for_change_of_measure_minus_one_minus_CLMG_Ls}
\lim_{\Vrt{v}_V\to 0}\frac{ \Vrt{ \mathcal{E}(M^v)_\tau-1-M^v_\tau}_{L^s(\mu^{u})}}{\Vrt{v}_V}=  0.
\end{equation}
If \eqref{eq:convergence_lemma_for_change_of_measure_minus_one_minus_CLMG_Ls} is true, then since $\Vrt{v}_V\leq 1$ implies that for every $0\leq r<1$, we have $\Vrt{v}_V^{-1} \geq \Vrt{v}_V^{-r}$ and hence
\begin{equation*}
\lim_{\Vrt{v}_V\to 0}\frac{ \Vrt{ \mathcal{E}(M^v)_\tau-1-M^v_\tau}_{L^s(\mu^{u})}}{\Vrt{v}_V^{r}}=  0.
\end{equation*}
Since \cref{asmp:uniform_lower_bound_on_smallest_singular_value_of_diffusion} holds, we may use \eqref{eq:Lr_bound_clmg} from \cref{lem:uniform_bound_on_quadratic_variation_time_t} and the triangle inequality to obtain
\begin{align*}
 \Vrt{ \mathcal{E}(M^v)_\tau-1}_{L^s(\mu^{u})}\leq &\Vrt{ \mathcal{E}(M^v)_\tau-1-M^v_\tau}_{L^s(\mu^{u})}+\Vrt{M^v_\tau}_{L^s(\mu^{u})}
 \\
 \leq &\Vrt{ \mathcal{E}(M^v)_\tau-1-M^v_\tau}_{L^s(\mu^{u})}+2\alpha^{-1}\Vrt{v}_V\Vrt{\tau}_{L^{s/2}(\mu^{u})}^{1/2}.
\end{align*}
By \cref{asmp:ball_of_admissible_changes_of_drift}, there exists $\lambda_u>0$ such that $\bb{E}^{u}[\exp(\lambda_u\tau)]$ is finite, and hence for every $s\in\bb{N}$, $\Vrt{\tau}_{L^{s/2}(\mu^{u})}$ is finite. Thus for any $0\leq r<1$,
\begin{equation*}
 \lim_{\Vrt{v}_V\to 0} \frac{\Vrt{ \mathcal{E}(M^v)_\tau-1}_{L^s(\mu^{u})}}{\Vrt{v}_V^{r}}=0,
\end{equation*}
which is \eqref{eq:convergence_lemma_for_change_of_measure_minus_one}. Thus, to prove \cref{lem:convergence_lemmas_for_change_of_measure}, it suffices to prove \eqref{eq:convergence_lemma_for_change_of_measure_minus_one_minus_CLMG_Ls}.

We shall use the following consequence of Young's inequality: for $x,y>0$ and $s\in\bb{N}$,
\begin{equation*}
 (x+y)^s<(1+2^{s-1})(x^s+y^s).
\end{equation*}
For a proof, see \cite[Lemma B.4]{LieStahnSullivan2021}. Using the inequality above, \eqref{eq:exponential_martingale}, and the series expansion for the exponential,
\begin{equation*}
\begin{aligned}
&\Vrt{ \mathcal{E}(M^v)_\tau-1-M^v_\tau}^s_{L^s(\mu^{u})}
\\
&\leq (1+2^{s-1})\bb{E}^{u}\left[ \frac{1}{2^s}\Ang{M^v,M^v}^s_\tau +\left(\sum_{\ell=2}^{\infty}\frac{1}{\ell!}\left(\abs{M^v_\tau}+\frac{1}{2}\Ang{M^v,M^v}_\tau\right)^\ell\right)^s\right].
\end{aligned}
\end{equation*}
By the bound \eqref{eq:Lr_bound_covariance} in \cref{lem:uniform_bound_on_quadratic_variation_time_t},
\begin{equation*}
 \lim_{\Vrt{v}_V\to 0}\frac{\bb{E}^{u}[\Ang{M^v,M^v}_\tau^s]}{\Vrt{v}_V^s}\leq \alpha^{-2}\Vrt{\tau}_{L^{s}(\mu^{u})}\lim_{\Vrt{v}_V\to 0}\frac{\Vrt{v}_V^{2s}}{\Vrt{v}_V^s}=0.
\end{equation*}
Hence, to prove \eqref{eq:convergence_lemma_for_change_of_measure_minus_one_minus_CLMG_Ls}, we must prove that
\begin{equation}
\label{eq:convergence_lemma_for_change_of_measure_minus_one_minus_CLMG_toprove}
 \lim_{\Vrt{v}_V\to 0}\frac{1}{\Vrt{v}_V^s}\bb{E}^{u}\left[\left(\sum_{\ell=2}^{\infty}\frac{1}{\ell!}\left(\abs{M^v_\tau}+\frac{1}{2}\Ang{M^v,M^v}_\tau\right)^\ell\right)^s\right]=0.
\end{equation}
For $a,b\geq 0$, collecting powers of $(a+b)$ and using the binomial theorem yields
\begin{align}
\left(\sum^\infty_{\ell=2}\frac{1}{\ell!}(a+b)^\ell\right)^s=&\sum_{n=2^s}^\infty\sum_{(i_j)_j\in\mathcal{I}_{s,n}}\frac{1}{\prod_{j\in[s]}i_j!}(a+b)^n
\nonumber
\\
=& \sum_{n=2^s}^\infty\sum_{(i_j)_j\in\mathcal{I}_{s,n}}\frac{1}{\prod_{j\in[s]}i_j!}\sum_{m=0}^{n}\begin{pmatrix} n \\ m \end{pmatrix}a^{n-m}b^{m},
\label{eq:sth_power_identity}
\end{align}
where $\mathcal{I}_{s,n}=\{(i_1,\ldots,i_s)\ :\ \forall j\in[s], i_j\geq 2,\ \sum_{j\in[s]} i_j=n\}$. 

Let $a:= \abs{M^v_\tau}$ and $b:=\Ang{M^v,M^v}_\tau$. Using Tonelli's theorem and \eqref{eq:sth_power_identity}, we obtain
\begin{equation}
\begin{aligned}
&\bb{E}^{u}\left[\left(\sum^\infty_{\ell=2}\frac{1}{\ell!}(\abs{M^v_\tau}+\Ang{M^v,M^v}_\tau)^\ell\right)^s\right] 
\\
=&\sum_{n=2^s}^\infty\sum_{(i_j)_j\in\mathcal{I}_{s,n}}\sum_{m=0}^{n}\begin{pmatrix} n \\ m \end{pmatrix}\bb{E}^{u}\left[\abs{M^v_\tau}^{n-m}\Ang{M^v,M^v}_\tau^{m}\right].
\label{eq:intermediate00}
\end{aligned}
\end{equation}
For the third sum on the right-hand side of \eqref{eq:intermediate00}, consider the summands corresponding to $0<m<n$. Observe that $p:=\tfrac{n+m}{n-m}>1$ and $q:= \tfrac{n+m}{2m}>1$ are H\"{o}lder conjugate exponents. Young's inequality yields
\begin{equation*}
 \abs{M^v_\tau}^{n-m}\Ang{M^v,M^v}_\tau^{m}\leq \frac{1}{p}\abs{M^v_\tau}^{(n-m)p}+\frac{1}{q}\Ang{M^v,M^v}_\tau^{m q}< \abs{M^v_\tau}^{n+m}+\Ang{M^v,M^v}_\tau^{(n+m)/2}.
\end{equation*}
	Hence, using Doob's inequality \eqref{eq:doobs_inequality}, it follows that for some absolute constant $C>0$,
	\begin{equation*}
		\bb{E}^{u}\left[\abs{M^v_\tau}^{n-m}\Ang{M^v,M^v}_\tau^{m}\right]\leq (C+1)\bb{E}^{u}\left[\Ang{M^v,M^v}_\tau^{(n+m)/2}\right],\quad 0<m<n.
	\end{equation*}
	For the summand given by $m=n$, $\bb{E}^{u}\left[\Ang{M^v,M^v}_\tau^n\right]= \bb{E}^{u}\left[\Ang{M^v,M^v}^{(n+n)/2}\right]$. For the summand given by $m=0$, \eqref{eq:doobs_inequality} yields $\bb{E}^{u}\left[\abs{M^v_\tau}^{n}\right]\leq C \bb{E}^{u}\left[\Ang{M^v,M^v}_\tau^{n/2}\right]$. Thus, the above inequality holds for $0\leq m\leq n$. Using this fact and the binomial theorem,
	\begin{align}
	 \sum_{m=0}^{n}\begin{pmatrix} n \\ m \end{pmatrix}\bb{E}^{u}\left[\abs{M^v_\tau}^{n-m}\Ang{M^v,M^v}_\tau^{m}\right]
    &\leq (C+1)\bb{E}^{u}\left[\Ang{M^v,M^v}_\tau^{n/2}\sum_{m=0}^{n}\begin{pmatrix} n \\ m \end{pmatrix}1^{n-m}\Ang{M^v,M^v}_\tau^{m/2}\right]
    \nonumber
    \\
    &= (C+1)\bb{E}^{u}\left[\Ang{M^v,M^v}_\tau^{n/2}\left(1+\Ang{M^v,M^v}_\tau^{1/2}\right)^n\right]
    \nonumber
    \\
    &= (C+1)\bb{E}^{u}\left[(\Ang{M^v,M^v}_\tau^{1/2}+\Ang{M^v,M^v}_\tau)^n\right].
    \nonumber
	\end{align}
    By using the above bound in \eqref{eq:intermediate00}, and then applying Tonelli's theorem and \eqref{eq:sth_power_identity}, we obtain 
    \begin{equation*}
     \bb{E}^{u}\left[\left(\sum^\infty_{\ell=2}\frac{1}{\ell!}(\abs{M^v_\tau}+\Ang{M^v,M^v}_\tau)^\ell\right)^s\right]
     \leq (C+1)\bb{E}^{u}\left[\left(\sum^\infty_{\ell=2}\frac{1}{\ell!}(\Ang{M^v,M^v}_\tau^{1/2}+\Ang{M^v,M^v}_\tau)^\ell\right)^s\right].
    \end{equation*}
    We will show that the right-hand side divided by $\Vrt{v}_V^s$ remains finite in the limit as $\Vrt{v}_V\to 0$. This will justify using the dominated convergence theorem to interchange the limit with the expectation, and thus will prove \eqref{eq:convergence_lemma_for_change_of_measure_minus_one_minus_CLMG_toprove}. For $0<\Vrt{v}_V\leq 1$, we have
    \begin{align*}
     \Ang{M^v,M^v}_\tau^{1/2}+\Ang{M^v,M^v}_\tau\leq (\alpha^{-2}\tau\Vrt{v}_V^2)^{1/2}+\alpha^{-2}\tau \Vrt{v}_V^2< \Vrt{v}_V+2\alpha^{-2}\tau\Vrt{v}_V.
    \end{align*}
    The first inequality follows from \cref{lem:uniform_bound_on_quadratic_variation_time_t}. The second inequality follows from applying $ab\leq \tfrac{1}{2}(a^2+b^2)$ with $a=\Vrt{v}_V^{1/2}$ and $b=(\alpha^{-2}\tau\Vrt{v}_V)^{1/2}$, and using that $\Vrt{v}_V\leq 1$ is equivalent to $\Vrt{v}_V^2\leq \Vrt{v}_V$. This implies the first inequality below:
    \begin{align*}
     &\frac{1}{\Vrt{v}_V^s}\bb{E}^{u}\left[\left(\sum^\infty_{\ell=2}\frac{1}{\ell!}(\Ang{M^v,M^v}_\tau^{1/2}+\Ang{M^v,M^v}_\tau)^\ell\right)^s\right] < \bb{E}^{u}\left[ \left(\frac{1}{\Vrt{v}_V}\sum^\infty_{\ell=2}\frac{\Vrt{v}_V^\ell}{\ell!}( 1+2\alpha^{-2}\tau)^\ell\right)^s\right]
     \\
     &\leq \bb{E}^{u}\left[ \left(\sum^\infty_{\ell=2}\frac{\Vrt{v}_V^{\ell/2}}{\ell!}(1+2\alpha^{-2} \tau)^\ell\right)^s\right]
     \\
     &< \bb{E}^{u}\left[\exp\left(s\Vrt{v}_V^{1/2}\left(1+2\alpha^{-2}\tau\right)\right)\right]=\exp\left(s\Vrt{v}_V^{1/2}\right)\bb{E}^{u}\left[\exp\left(2\alpha^{-2}s\Vrt{v}_V^{1/2}\tau\right)\right].
    \end{align*}
    The second inequality follows since $0<\Vrt{v}_V\leq 1$ implies that for every $\ell\geq 2$, $\Vrt{v}_V^{-1}\leq \Vrt{v}_V^{-\ell/2}$. The third inequality follows from the series expansion of the exponential. 
   By \cref{asmp:ball_of_admissible_changes_of_drift}, there exists $\lambda_u>0$ such that $\bb{E}^{u}[\exp(\lambda_u\tau)]$ is finite. Thus, if $2\alpha^{-2}s\Vrt{v}_V^{1/2}\leq \lambda_u$, then the last term is finite, and we may apply the dominated convergence theorem to prove \eqref{eq:convergence_lemma_for_change_of_measure_minus_one_minus_CLMG_toprove}, as desired.
\end{proof}

Below, we shall use the following generalisation of the bound \eqref{eq:Lr_bound_clmg} from \cref{lem:uniform_bound_on_quadratic_variation_time_t}, which uses \cref{asmp:uniform_lower_bound_on_smallest_singular_value_of_diffusion}: for $r,s\geq 1$,
\begin{equation}
\label{eq:Lr_bound_clmg_to_power_s}
 \Vrt{(M^w_\tau)^s}_{L^r(\mu^{u})}=\Vrt{M^w_\tau}_{L^{rs}(\mu^{u})}^{s}\leq \left(2\alpha^{-1}\Vrt{w}_V\right)^{s}\Vrt{\tau}^{s/2}_{L^{rs/2}(\mu^{u})}.
\end{equation}
For the rest of this section, we shall denote by $ Z$ a `placeholder' random variable. See \cref{rem:placeholder}. In particular, we may consider $ Z$ as being independent of $u$.

\begin{lemma}
 \label{lem:frechet_derivative_for_product_of_Mu_clmg_with_Z}
 Let \cref{asmp:uniform_lower_bound_on_smallest_singular_value_of_diffusion} and \cref{asmp:ball_of_admissible_changes_of_drift} hold, let $u\in U$, and let $Z$ be a random variable that has absolute moments of all orders with respect to $\mu^{u}$. Then the Fr\'{e}chet derivative of $U\ni u'\mapsto \bb{E}^{u'}[M^{u'}_\tau Z]$ at $u$ is given by
 \begin{equation*}
  V\ni v\mapsto \bb{E}^{u}\left[\left(M^{u}_\tau+1\right) Z M^{v}_{\tau}\right].
 \end{equation*}
\end{lemma}
\begin{proof}
 By \eqref{eq:change_of_measure} and \eqref{eq:clmg},
\begin{equation*}
  \bb{E}^{u+w}[M^{u+w}_{\tau} Z ]=\bb{E}^{u}\left[\mathcal{E}(M^w)_\tau\left(M^u_\tau+M^w_\tau\right) Z\right].
 \end{equation*}
 Hence,
 \begin{align}
  &\abs{\bb{E}^{u+w}\left[M^{u+w}_{\tau}Z\right]-\bb{E}^{u}\left[M^u_\tau Z\right]-\bb{E}^{u}\left[\left(M^{u}_\tau+1\right)ZM^{w}_{\tau} \right]}
  \nonumber
  \\
  =&\abs{\bb{E}^{u}\left[ M^u_\tau Z\left(\mathcal{E}(M^w)_\tau-1-M^w_\tau\right)+ M^w_\tau Z \left(\mathcal{E}(M^w)_\tau-1\right)\right] }
  \nonumber
  \\
  \leq &\Vrt{M^u_\tau Z}_{L^2(\mu^{u})}\Vrt{\mathcal{E}(M^w)_\tau-1-M^w_\tau}_{L^2(\mu^{u})}+\Vrt{M^w_\tau Z}_{L^2(\mu^{u})}\Vrt{\mathcal{E}(M^w)_\tau-1}_{L^2(\mu^{u})}.
  \label{eq:intermediate10}
  \end{align}
 Now by H\"{o}lder's inequality and \eqref{eq:Lr_bound_clmg},
 \begin{equation*}
  \Vrt{M^u_\tau Z}_{L^2(\mu^{u})}\leq \Vrt{M^u_\tau}_{L^4(\mu^{u})}\Vrt{Z}_{L^4(\mu^{u})}\leq  2\alpha^{-1}\Vrt{u}_V\Vrt{\tau}_{L^2(\mu^{u})}^{1/2}\Vrt{Z}_{L^4(\mu^{u})}
 \end{equation*}
 and a corresponding bound holds for $\Vrt{M^w_\tau Z}_{L^2(\mu^{u})}$. Thus, the terms that do not involve $\mathcal{E}(M^w)_\tau$ are finite, since $\tau$ has moments of all orders by \cref{asmp:ball_of_admissible_changes_of_drift}. Divide the inequality \eqref{eq:intermediate10} by $\Vrt{w}_V$ and let $\Vrt{w}_V\to 0$. Then the term corresponding to the first term on the right-hand side of \eqref{eq:intermediate10} vanishes by \eqref{eq:convergence_lemma_for_change_of_measure_minus_one_minus_CLMG}. The term corresponding to the remaining term vanishes by \eqref{eq:convergence_lemma_for_change_of_measure_minus_one}, because
 \begin{equation*}
  \frac{\Vrt{M^w_\tau Z}_{L^2(\mu^{u})}\Vrt{\mathcal{E}(M^w)_\tau-1}_{L^2(\mu^{u})}}{\Vrt{w}_V}\leq 2\alpha^{-1}\Vrt{\tau}_{L^2(\mu^{u})}^{1/2}\frac{ \Vrt{w}_V\Vrt{\mathcal{E}(M^w)_\tau-1}_{L^2(\mu^{u})}}{\Vrt{w}_V}.
 \end{equation*}
 Linearity of the derivative follows from the linearity of $v\mapsto M^v_\tau$ in \eqref{eq:clmg}, while boundedness follows from the fact that
 \begin{equation*}
  \abs{ \bb{E}^{u}\left[\left(M^{u}_\tau+1\right) Z M^{v}_{\tau}\right]}\leq \left(\Vrt{M^u_\tau}_{L^3(\mu^{u})}+1\right)\Vrt{Z}_{L^3(\mu^{u})}\Vrt{M^v_\tau}_{L^3(\mu^{u})}
 \end{equation*}
and using \eqref{eq:Lr_bound_clmg}.
\end{proof}

\begin{lemma}
 \label{lem:frechet_derivative_for_KLD_mu_u_bar_mu_0_precursor}
 Let \cref{asmp:uniform_lower_bound_on_smallest_singular_value_of_diffusion} and \cref{asmp:ball_of_admissible_changes_of_drift} hold, let $u\in U$, and let $Z$ be a random variable that has absolute moments of all orders with respect to $\mu^{u}$. Then the Fr\'{e}chet derivative of $U\ni u'\mapsto \bb{E}^{u'}[(M^{u'}_\tau)^2 Z]$ at $u$ is given by
 \begin{equation*}
  V\ni v\mapsto \bb{E}^{u}\left[\left(\left(M^{u}_\tau\right)^2+2M^{u}_{\tau}\right) Z M^{v}_{\tau}\right].
 \end{equation*}
\end{lemma}
\begin{proof}
 By \eqref{eq:change_of_measure} and \eqref{eq:clmg},
 \begin{equation*}
  \bb{E}^{u+w}[(M^{u+w}_{\tau})^2 Z ]=\bb{E}^{u}\left[\mathcal{E}(M^w)_\tau\left(\left(M^u_\tau\right)^2+2 M^{u}_{\tau}M^{w}_{\tau}+(M^{w}_{\tau})^2\right) Z\right].
 \end{equation*}
 Hence,
 \begin{align*}
  &\abs{\bb{E}^{u+w}\left[(M^{u+w}_{\tau})^2Z\right]-\bb{E}^{u}\left[(M^u_\tau)^2 Z\right]-\bb{E}^{u}\left[\left(\left(M^{u}_\tau\right)^2+2M^{u}_{\tau}\right)M^{v}_{\tau} Z\right]}
  \\
  =&\abs{\bb{E}^{u}\left[ \left(M^u_\tau\right)^2 Z\left(\mathcal{E}(M^w)_\tau-1-M^w_\tau\right)+2 M^u_\tau M^w_\tau Z \left(\mathcal{E}(M^w)_\tau-1\right)+(M^w_\tau)^2 Z\mathcal{E}(M^w)_\tau\ \right] }
  \\
  \leq &\Vrt{(M^u_\tau)^2 Z}_{L^2(\mu^{u})}\Vrt{\mathcal{E}(M^w)_\tau-1-M^w_\tau}_{L^2(\mu^{u})}+2\Vrt{M^u_\tau M^w_\tau Z}_{L^2(\mu^{u})}\Vrt{\mathcal{E}(M^w)_\tau-1}_{L^2(\mu^{u})}
  \\
  &+\Vrt{(M^w_\tau)^2 Z}_{L^2(\mu^{u})}\Vrt{\mathcal{E}(M^w)_\tau}_{L^2(\mu^{u})}.
 \end{align*}
  Using H\"{o}lder's inequality and \eqref{eq:Lr_bound_clmg_to_power_s}, all the terms on the right-hand side of the inequality that do not involve $\mathcal{E}(M^w)_\tau$ are finite, since \cref{asmp:ball_of_admissible_changes_of_drift} implies that $\tau$ has moments of all orders with respect to $\mu^{u}$, and since $Z$ has absolute moments of all orders. Divide the right-hand side of the inequality by $\Vrt{w}_V$, let $\Vrt{w}_V\to 0$, and consider the limiting behaviour of the three terms. The terms corresponding to the first and second terms on the right-hand side of the inequality vanish by \eqref{eq:convergence_lemma_for_change_of_measure_minus_one_minus_CLMG} and by \eqref{eq:convergence_lemma_for_change_of_measure_minus_one} respectively. The third term vanishes by using \eqref{eq:Lr_bound_clmg_to_power_s} with $r=s=2$.
Linearity of the map above follows from linearity of the map $v\mapsto M^v_\tau$ in \eqref{eq:clmg}. Boundedness follows by using
 \begin{equation*}
  \abs{\bb{E}^{u}\left[\left(\left(M^{u}_\tau\right)^2+2M^{u}_{\tau}\right)ZM^{v}_{\tau}\right]}\leq \left(\Vrt{(M^u_\tau)^2}_{L^3(\mu^{u})}+2\Vrt{M^u_\tau}_{L^3(\mu^{u})}\right)\Vrt{Z}_{L^3(\mu^{u})}\Vrt{M^v_\tau}_{L^3(\mu^{u})}
 \end{equation*}
 and \eqref{eq:Lr_bound_clmg_to_power_s}. 
\end{proof}

Recall that \cref{lem:frechet_derivative_for_KLD_mu_u_bar_mu_0_clmg_form} states that if  \cref{asmp:uniform_lower_bound_on_smallest_singular_value_of_diffusion} and \cref{asmp:ball_of_admissible_changes_of_drift} hold, then the Fr\'{e}chet derivative of $U\ni u'\mapsto \bb{E}^{u'}[(M^{u'}_\tau)^2]$ at $u$ is given by
 \begin{equation}
  \label{eq:frechet_derivative_for_KLD_mu_u_bar_mu_0}
   V\ni v\mapsto \bb{E}^{u}\left[\left(\left(M^{u}_\tau\right)^2+2M^{u}_{\tau}\right)M^{v}_{\tau}\right].
 \end{equation}

\begin{proof}[Proof of \cref{lem:frechet_derivative_for_KLD_mu_u_bar_mu_0_clmg_form}]
 Apply \cref{lem:frechet_derivative_for_KLD_mu_u_bar_mu_0_precursor} with $Z$ being the constant random variable 1.
\end{proof}

Next, recall that \cref{lem:mixed_second_order_frechet_derivative_for_KLD_mu_u_bar_mu_0_clmg_form} states that if \cref{asmp:uniform_lower_bound_on_smallest_singular_value_of_diffusion} and \cref{asmp:ball_of_admissible_changes_of_drift} hold, then for $n\in\bb{N}\setminus\{1\}$, the $n$-th order Fr\'{e}chet derivative of $U\ni u'\mapsto \bb{E}^{u'}[(M^{u'}_\tau)^2]$ at $u$ is given by
\begin{equation*}
 \times_{k=1}^{n}V\ni (v_1,\ldots,v_n)\mapsto \bb{E}^{u}\left[ \left(\left(M^u_\tau\right)^2+2n M^u_\tau +n(n-1)\right)\prod_{k=1}^{n}M^{v_k}_\tau\right].
\end{equation*}
 for $n\in\bb{N}$.
\begin{proof}[Proof of \cref{lem:mixed_second_order_frechet_derivative_for_KLD_mu_u_bar_mu_0_clmg_form}]
For the proof below, we shall use the fact that $2\sum_{i=0}^{n-1}i=n(n-1)$, and we shall use induction. The base case where $n=1$ follows from \cref{lem:frechet_derivative_for_KLD_mu_u_bar_mu_0_clmg_form}. For the inductive step, assume that the statement is true for $n\leq 1$. Since the Fr\'{e}chet derivative is a linear operator, we may apply \cref{lem:frechet_derivative_for_KLD_mu_u_bar_mu_0_precursor} with $Z=\prod_{k=1}^{n}M^{v_k}_\tau$ to conclude that the Fr\'{e}chet derivative of $u'\mapsto \bb{E}^{u'}[(M^{u'}_\tau)^2\prod_{k=1}^{n} M^{v_k}_\tau]$ at $u$ is 
\begin{equation*}
V\ni v_{n+1}\mapsto \bb{E}^{u}\left[\left( \left(M^u_\tau\right)^2+2M^u_\tau\right)\prod_{k=1}^{n+1}M^{v_k}_\tau\right].
\end{equation*}
Next, we may apply \cref{lem:frechet_derivative_for_product_of_Mu_clmg_with_Z} with $Z=\prod_{k=1}^{n}M^{v_k}_\tau$ to conclude that the Fr\'{e}chet derivative of $u'\mapsto 2n\bb{E}^{u'}[M^{u'}_\tau\prod_{k=1}^{n} M^{v_k}_\tau]$ at $u$ is 
\begin{equation*}
V\ni v_{n+1}\mapsto 2n\bb{E}^{u}\left[\left( M^u_\tau+1\right)\prod_{k=1}^{n+1}M^{v_k}_\tau\right].
\end{equation*}
Then, we apply \cref{lem:frechet_derivative_for_functional_not_depending_on_perturbation} with $\phi_\tau=\left(2\sum_{i=0}^{n-1}i\right)\prod_{k=1}^{n} M^{v_k}_\tau$ to conclude that the Fr\'{e}chet derivative of $u'\mapsto \left(2\sum_{i=0}^{n-1}i\right)\bb{E}^{u'}[\prod_{k=1}^{n} M^{v_k}_\tau]$ at $u$ is 
\begin{equation*}
 V\ni v_{n+1}\mapsto \left(2\sum_{i=0}^{n-1}i\right)\bb{E}^{u}\left[\prod_{k=1}^{n+1} M^{v_k}_\tau\right].
\end{equation*}
Summing these derivatives yields
\begin{align*}
 V\ni v_{n+1}\mapsto &\bb{E}^{u}\left[\left( \left(M^u_\tau\right)^2+2M^u_\tau+2n(M^u_\tau+1)+2\sum_{i=0}^{n-1}i\right)\prod_{k=1}^{n+1}M^{v_k}_\tau\right]
 \\
 =&\bb{E}^{u}\left[\left( \left(M^u_\tau\right)^2+2(n+1)M^u_\tau+2\sum_{i=0}^{n}i\right)\prod_{k=1}^{n+1}M^{v_k}_\tau\right],
\end{align*}
and thus completes the inductive step.
\end{proof}

\bibliographystyle{amsplain}
\bibliography{references}

\end{document}